\documentclass[10pt,authoryear]{article}%
\usepackage{amsmath}
\usepackage{amsfonts}
\usepackage{mathrsfs}
\usepackage{amssymb,bbm, color}
\usepackage[hidelinks,colorlinks=true,linkcolor=blue,citecolor=blue]{hyperref}
\usepackage{graphicx}
\numberwithin{equation}{section}
\usepackage[body={15cm,20cm}, top=3cm]{geometry}
\usepackage{enumitem} 
\providecommand{\U}[1]{\protect\rule{.1in}{.1in}}
\providecommand{\U}[1]{\protect \rule{.1in}{.1in}}
\newtheorem{theorem}{Theorem}[section]

\newtheorem{definition}[theorem]{Definition}

\newtheorem{lemma}[theorem]{Lemma}

\newtheorem{proposition}[theorem]{Proposition}
\newtheorem{remark}[theorem]{Remark}

\newenvironment{proof}[1][Proof]{\noindent \textbf{#1.} }{\  \rule{0.5em}{0.5em}}

\def \hE {\widehat{\mathbb E}}

 \usepackage{natbib}
\allowdisplaybreaks

\begin{document}
	\title{Doubly Reflected Backward SDEs Driven by $G$-Brownian Motions and Fully Nonlinear PDEs with Double Obstacles}
	\author{Hanwu Li\thanks{Research Center for Mathematics and Interdisciplinary Sciences, Shandong University, Qingdao 266237, Shandong, China. Email: lihanwu@sdu.edu.cn.}
	\thanks{Frontiers Science Center for Nonlinear Expectations (Ministry of Education), Shandong University, Qingdao 266237, Shandong, China.}
	\and Ning Ning\thanks{Department of Statistics,
		Texas A\&M University, College Station, Texas, USA. Email: patning@tamu.edu.}}
	\date{}
	\maketitle
	\begin{abstract}
	In this paper, we introduce a new method to study the doubly reflected  backward stochastic differential equation driven by $G$-Brownian motion ($G$-BSDE). Our approach involves approximating the solution through a family of penalized reflected $G$-BSDEs with a lower obstacle that are monotone decreasing. By employing this approach, we establish the well-posedness of the solution of the doubly reflected $G$-BSDE with the weakest known conditions, and uncover its relationship with the fully nonlinear partial differential equation with double obstacles for the first time.
	\end{abstract}
	
	
	\textbf{MSC-classification}: 60H10, 60H30
	
	
	\section{Introduction}
We firstly give the background in Subsection \ref{sec:Background} and then state our contributions in Subsection \ref{sec:contributions}, followed with the organization of the paper in Subsection \ref{sec:organization}.

	\subsection{Background} 
\label{sec:Background}	
In 1997, \cite{el1997reflected} first introduced the reflected backward stochastic differential equation (RBSDE), where the first component of the solution is constrained to remain above a specified continuous process, known as the obstacle. To enforce this constraint, an additional non-decreasing process is introduced to push the solution upwards, while adhering to the Skorohod condition in a minimal manner. This problem is intimately linked to various fields including optimal stopping problems (see, e.g., \cite{cheng2013optimal}), pricing for American options (see, e.g., \cite{el1997reflected2}), and the obstacle problem for partial differential equations (PDEs) (see, e.g., \cite{bally2002reflected}).

Subsequently, \cite{cvitanic1996backward} extended the above results to encompass scenarios involving two obstacles. In this setting, the solution 
$Y$ is constrained to remain between two specified continuous processes, known as the lower and upper obstacles. Consequently, two non-decreasing processes are introduced in the doubly RBSDE, with the aim of pushing the solution upwards and pulling it downwards, respectively, while ensuring adherence to the Skorohod conditions. Additionally, they demonstrated that the solution 
$Y$ coincides with the value function of a Dynkin game. Given its significance in both theoretical analysis and practical applications, numerous studies have been conducted. Interested readers may refer to  \cite{crepey2008reflected,dumitrescu2016generalized,grigorova2018doubly,hamadene2005bsdes,hamadene1997double,peng2005smallest} and the references therein for further exploration.

The classical theory is limited to solving financial problems under drift uncertainty and the associated semi-linear PDEs. Motivated by the need to address financial problems under volatility uncertainty and the associated fully nonlinear PDEs,  \cite{peng2007g,peng2008multi,peng2019nonlinear} introduced a novel nonlinear expectation theory known as the $G$-expectation theory. This theory involves the construction of a nonlinear Brownian motion, termed $G$-Brownian motion, and the introduction of corresponding $G$-It\^{o}'s calculus. Building upon the $G$-expectation theory, \cite{hu2014backward} investigated BSDEs driven by $G$-Brownian motions ($G$-BSDEs). In comparison with classical results, $G$-BSDEs include an additional non-increasing $G$-martingale $K$ in the equation due to nonlinearity. In \cite{hu2014backward}, the authors established the well-posedness of $G$-BSDEs, while the comparison theorem, Feynman-Kac formula, and Girsanov transformation can be found in their companion paper \cite{hu2014comparison}.

In recent years, \cite{li2018reflected} introduced reflected $G$-BSDEs with a lower obstacle. Given the presence of a non-increasing $G$-martingale in $G$-BSDEs, the definition deviates from the classical case. Specifically, they amalgamated the non-decreasing process, intended to elevate the solution, with the non-increasing $G$-martingale into a general non-decreasing process that satisfies a martingale condition. Existence was established through approximation via penalization, while uniqueness was derived from a prior estimates. 
For further insights, readers may refer to \cite{li2017reflected}.
The study of reflected $G$-BSDEs with two obstacles is undertaken by \cite{li2021backward}. They introduced a so-called approximate Skorohod condition and established the well-posedness of doubly reflected $G$-BSDEs when the upper obstacle is a generalized $G$-It\^{o}  process.

\subsection{Our contributions}
\label{sec:contributions}

Three natural questions arise concerning the doubly reflected $G$-BSDE of the following form:
\begin{align*}
	\begin{cases}
		Y_t=\xi+\int_t^T f(s,Y_s,Z_s)ds+\int_t^T g(s,Y_s,Z_s)d\langle B\rangle_s-\int_t^T Z_s dB_s+(A_T-A_t),\vspace{0.2cm}\\
		L_t\leq Y_t\leq U_t,\quad 0\leq t\leq T, \vspace{0.2cm}\\
		(Y, A) \textrm{ satisfies the approximate Skorohod condition with order $\alpha$ } (\textmd{ASC}_{\alpha}),
	\end{cases}
\end{align*}
whose detailed description and the $\textmd{ASC}_{\alpha}$ are provided in Subsection \ref{sec:DRGBSDE}.
 Firstly, what types of fully nonlinear PDEs can be represented by reflected $G$-BSDEs with two obstacles? Secondly, in addition to finding that connection, can we enhance the theory of doubly reflected $G$-BSDEs further at the same time? Lastly, in order to achieve both of these goals, what new mathematical strategies suffice and can be developed? The objective of this paper is to address these three questions.

We discovered that to establish the connection between the solution of the doubly reflected $G$-BSDE and the double obstacle fully nonlinear PDEs, it would be desirable if we can construct a monotone sequence  converging to that solution. However, achieving this is challenging. Say, we consider the following penalized reflected $G$-BSDEs with a lower obstacle parameterized by $n\in\mathbb{N}$, which has been a workhorse in existing literature:
\begin{align*}
	\begin{cases}
		&\bar{Y}^n_t=\xi+\int_t^T f(s,\bar{Y}^n_s,\bar{Z}^n_s)ds+\int_t^T g(s,\bar{Y}^n_s,\bar{Z}^n_s)d\langle B\rangle_s-n\int_t^T(\bar{Y}^n_s-U_s)^+ds\\
		&\hspace{1cm}-\int_t^T \bar{Z}^n_sdB_s+(\bar{A}^n_T-\bar{A}^n_t),\vspace{0.2cm}\\
		&\bar{Y}^n_t\geq L_t,\quad 0\leq t\leq T, \vspace{0.2cm}\\
		&\big\{-\int_0^t (\bar{Y}^n_s-L_s)d\bar{A}^n_s\big\}_{t\in[0,T]} \textrm{ is a non-increasing $G$-martingale}.
	\end{cases}
\end{align*}
By the comparison theorem for reflected $G$-BSDEs, $\bar{Y}^n$ is non-increasing in $n$. The purpose of the penalization term is to drive the solution $\bar{Y}^n$ downwards so that the limiting process $Y$ (if it exists) remains below $U$. Thus, the remaining challenge is to demonstrate that the sequence $\bar{Y}^n$ converges to  some process $Y$, which is the first component of solution to the desired doubly reflected $G$-BSDE. However, unlike in \cite{li2017reflected} and \cite{li2021backward}, the main problem is that $\bar{A}^n$ is no longer a $G$-martingale. Consequently, we are unable to show that $(\bar{Y}^n-U)^+$ converges to $0$ with the explicit rate $\frac{1}{n}$.

Then we found that for $n,m\in\mathbb{N}$, we could consider the following family of $G$-BSDEs instead:
\begin{displaymath}\begin{split}
		Y^{n,m}_t=&\xi+\int_t^T f(s,Y^{n,m}_s,Z^{n,m}_s)ds+\int_t^T g(s,Y^{n,m}_s,Z^{n,m}_s)d\langle B\rangle_s-\int_t^T Z_s^{n,m}dB_s\\
		&+\int_t^T m(Y_s^{n,m}-L_s)^-ds-\int_t^T n(Y_s^{n,m}-U_s)^+ds-(K_T^{n,m}-K_t^{n,m}),
\end{split}\end{displaymath}
and set $A^{n,m,+}_t=\int_0^t m(Y_s^{n,m}-L_s)^-ds$ and $A^{n,m,-}_t=\int_0^t n(Y_s^{n,m}-U_s)^+ds$. By letting $m$ tend to infinity, we can demonstrate that $(Y^{n,m},Z^{n,m}, A^{n,m,+}-K^{n,m})$ converges to $(\bar{Y}^n,\bar{Z}^n,\bar{A}^n)$. Then, as $n$ tends to infinity, $(\bar{Y}^n,\bar{Z}^n,\bar{A}^n)$ converges to $(Y,Z,A)$, the solution of the doubly reflected $G$-BSDE. 
Specifically, the penalized $G$-BSDEs with parameters $n$ and $m$ enable us to determine the convergence rate of $(Y^{n,m}-U)^+$, with an explicit rate of $\frac{1}{n}$, uniformly in $m$ (refer to Lemma \ref{Y^{n,m}-U}). Consequently, the convergence rate remains consistent for the limit process $(\bar{Y}^n-U)^+$. However, achieving uniform boundedness of $Y^{n,m}$ requires a different approach than the one used in \cite{li2021backward}. Therefore, we abandoned the application of $G$-It\^{o}'s formula and instead resorted to employing comparison results. Although this approach is not novel and dates back decades to \cite{peng2005smallest}, its application in the context of reflected $G$-BSDEs is innovative. Our first main result, establishing the well-posedness of doubly reflected $G$-BSDEs and their approximating sequences, is presented in Theorem \ref{main}.

Indeed, the approach that can be used to answer those three natural questions is considerably more intricate than the methods employed in  \cite{li2021backward}. However, the existence of the non-increasing $G$-martingale introduces a disparity between reflected $G$-BSDEs with upper and lower obstacles. The inclusion of both the non-increasing $G$-martingale and the non-increasing process for pulling down the solution results in a finite variation process, complicating the derivation of a priori estimates. Consequently, we made every effort to recycle results from \cite{li2021backward} and extend certain preliminary results. For example, Proposition \ref{a prior estimate} extends Proposition 3.1 in \cite{li2021backward} in two aspects. Both propositions aim to assess the difference between the first components of solutions to doubly reflected $G$-BSDEs. Notably, in Proposition \ref{a prior estimate}, the obstacles of the doubly reflected $G$-BSDEs are permitted to vary, while in Proposition 3.1 in \cite{li2021backward}, equality is assumed for the obstacles, i.e., $L^1\equiv L^2$ and $U^1\equiv U^2$, making it a special case of our condition. Moreover, our general conditions are even more relaxed. The advantage of this construction is that $\bar{Y}^n$ is non-increasing in $n$ and the solution $\bar{Y}^n$ provides a probabilistic representation for the PDE with an obstacle in a Markovian setting, which enable us to establish the connection between doubly reflected $G$-BSDEs and PDEs with two obstacles in the last section. Generally speaking, in a Markovian framework, the solution $Y$ of the doubly reflected $G$-BSDE is the unique viscosity solution of the associated double obstacle PDE, which extends the result in \cite{hamadene2005bsdes} to the fully nonlinear case. Our second main result, the function $u$ defined in \eqref{eq1.8} being the solution to the fully nonlinear obstacle problem \eqref{eq1.7}, is presented in Theorem \ref{the1.21}.

\subsection{Organization of the paper}
\label{sec:organization}
The remaining sections of the paper are structured as follows. In Section \ref{sec:Preliminaries}, we provide an overview of fundamental concepts and findings pertaining to $G$-expectation, $G$-BSDEs, and reflected $G$-BSDEs. In Section \ref{sec:Well-posedness}, we delve into the investigation of doubly reflected $G$-BSDEs and establish their well-posedness. In Section \ref{sec:PDE}, we establish the relationship between fully nonlinear PDEs with double obstacles and doubly reflected $G$-BSDEs.
Throughout the paper, the letter $C$, with or without subscripts, will represent a positive constant whose value may vary from line to line.

\section{Preliminaries}
\label{sec:Preliminaries}
We provide a brief overview of fundamental concepts and findings concerning $G$-expectation, $G$-BSDEs, and reflected $G$-BSDEs. To keep it concise, we focus solely on the one-dimensional case. For further elaboration, interested readers are encouraged to consult \cite{hu2014backward,hu2014comparison,li2018reflected,peng2007g,peng2008multi,peng2019nonlinear}.

\subsection{$G$-expectation and $G$-It\^{o}'s calculus}



Let $\Omega_T=C_{0}([0,T];\mathbb{R})$, the space of
real-valued continuous functions starting from the origin, i.e., $\omega_0=0$ for any $\omega\in \Omega_T$, be endowed
with the supremum norm. Let $\mathcal{B}(\Omega_T)$ be the Borel set and $B$ be the canonical process. Set
\[
L_{ip} (\Omega_T):=\Big\{ \varphi(B_{t_{1}},...,B_{t_{n}}):  \ n\in\mathbb {N}, \ t_{1}
,\cdots, t_{n}\in\lbrack0,T], \ \varphi\in C_{b,Lip}(\mathbb{R}^{ n})\Big\},
\]
where $C_{b,Lip}(\mathbb{R}^{ n})$ denotes the set of all bounded Lipschitz functions on $\mathbb{R}^{n}$.
We fix a sublinear and monotone function $G:\mathbb{R}\rightarrow\mathbb{R}$ defined by
\begin{align}\label{GG}
G(a):=\frac{1}{2}(\overline{\sigma}^2a^+-\underline{\sigma}^2a^-), 
\end{align}
where $0< \underline{\sigma}^2<\overline{\sigma}^2$. The associated $G$-expectation on $(\Omega_T, L_{ip}(\Omega_T))$ can be constructed in the following way. Given that $\xi\in L_{ip}(\Omega_T)$ can be represented as
$\xi=\varphi(B_{{t_1}}, B_{t_2},\cdots,B_{t_n})$,
set for $t\in[t_{k-1},t_k)$ with $k=1,\cdots,n$, 
\begin{displaymath}
	\widehat{\mathbb{E}}_{t}[\varphi(B_{{t_1}}, B_{t_2},\cdots,B_{t_n})]:=u_k(t, B_t;B_{t_1},\cdots,B_{t_{k-1}}),
\end{displaymath}
where $u_k(t,x;x_1,\cdots,x_{k-1})$ is a function of $(t,x)$ parameterized by $(x_1,\cdots,x_{k-1})$ such that it solves the following fully nonlinear PDE defined on $[t_{k-1},t_k)\times\mathbb{R}$:
\begin{displaymath}
	\partial_t u_k+G(\partial_x^2 u_k)=0,
\end{displaymath}
whose terminal conditions are given by
\begin{displaymath}
	\begin{cases}
	u_k(t_k,x;x_1,\cdots,x_{k-1})=u_{k+1}(t_k,x;x_1,\cdots,x_{k-1},x), \quad k<n,\\
	u_n(t_n,x;x_1,\cdots,x_{n-1})=\varphi(x_1,\cdots,x_{n-1},x).
	\end{cases}
\end{displaymath}
Hence, the $G$-expectation of $\xi$ is $\widehat{\mathbb{E}}_0[\xi]$, denoted as $\widehat{\mathbb{E}}[\xi]$ for simplicity. The triple $(\Omega_T, L_{ip}(\Omega_T),\widehat{\mathbb{E}})$ is called the $G$-expectation space and the process $B$ is the $G$-Brownian motion. 

For $\xi\in L_{ip}(\Omega_T)$ and $p\geq1$, we define
$$\Vert\xi\Vert_{L_{G}^{p}}:=(\widehat{\mathbb{E}}|\xi|^{p}])^{1/p}.$$
The completion of $L_{ip} (\Omega_T)$ under this norm  is denote by $L_{G}^{p}(\Omega)$.   For all $t\in[0,T]$, $\widehat{\mathbb{E}}_t[\cdot]$ is a continuous mapping on $L_{ip}(\Omega_T)$ w.r.t the norm $\|\cdot\|_{L_G^1}$. Hence, the conditional $G$-expectation $\mathbb{\widehat{E}}_{t}[\cdot]$ can be
extended continuously to the completion $L_{G}^{1}(\Omega_T)$. Furthermore, \cite{denis2011function} proved that the $G$-expectation has the following representation.
\begin{theorem}[\cite{denis2011function}]
	\label{the1.1}  There exists a weakly compact set
	$\mathcal{P}$ of probability
	measures on $(\Omega_T,\mathcal{B}(\Omega_T))$, such that
	\[
	\widehat{\mathbb{E}}[\xi]=\sup_{P\in\mathcal{P}}E_{P}[\xi], \qquad\forall \xi\in  {L}_{G}^{1}{(\Omega_T)}.
	\]
	We call $\mathcal{P}$ a set that represents $\widehat{\mathbb{E}}$.
\end{theorem}

For $\mathcal{P}$ being a weakly compact set that represents $\widehat{\mathbb{E}}$, we define the capacity%
\[
c(A):=\sup_{P\in\mathcal{P}}P(A),\qquad \forall A\in\mathcal{B}(\Omega_T).
\]
A set $A\in\mathcal{B}(\Omega_T)$ is called polar if $c(A)=0$.  A
property holds $``quasi$-$surely"$ (q.s.) if it holds outside a
polar set. In the sequel, we do not distinguish two random variables $X$ and $Y$ if $X=Y$, q.s..

\begin{definition}
	\label{def2.6} Let $M_{G}^{0}(0,T)$ be the collection of processes such that
	\[
	\eta_{t}(\omega)=\sum_{j=0}^{N-1}\xi_{j}(\omega)\mathbbm{1}_{[t_{j},t_{j+1})}(t),
	\]
	where $\xi_{i}\in L_{ip}(\Omega_{t_{i}})$ for a given partition $\{t_{0},\cdot\cdot\cdot,t_{N}\}$ of $[0,T]$. For each
	$p\geq1$ and $\eta\in M_G^0(0,T)$, denote $$\|\eta\|_{H_G^p}:=\left\{\widehat{\mathbb{E}}\bigg(\int_0^T|\eta_s|^2ds\bigg)^{p/2}\right\}^{1/p}\quad\text{and}\quad \Vert\eta\Vert_{M_{G}^{p}}:=\left\{\widehat{\mathbb{E}}\bigg(\int_{0}^{T}|\eta_{s}|^{p}ds\bigg)\right\}^{1/p}.$$ 
	Let $H_G^p(0,T)$ and $M_{G}^{p}(0,T)$ be the completions
	of $M_{G}^{0}(0,T)$ under the norms $\|\cdot\|_{H_G^p}$ and $\|\cdot\|_{M_G^p}$, respectively.
\end{definition}

Denote by $\langle B\rangle$ the quadratic variation process of the $G$-Brownian motion $B$. For two processes $ \xi\in M_{G}^{1}(0,T)$ and $ \eta\in M_{G}^{2}(0,T)$,
the $G$-It\^{o} integrals $(\int^{t}_0\xi_sd\langle
B\rangle_s)_{0\leq t\leq T}$ and $(\int^{t}_0\eta_sdB_s)_{0\leq t\leq T}$ are well defined, see  \cite{li2011stopping} and \cite{peng2019nonlinear}. The following proposition can be regarded as the Burkholder--Davis--Gundy inequality under the $G$-expectation framework.
\begin{proposition}[\cite{hu2014comparison}]\label{the1.3}
	If $\eta\in H_G^{\alpha}(0,T)$ with $\alpha\geq 1$ and $p\in(0,\alpha]$, then we have
	\begin{displaymath}
		\underline{\sigma}^p c\widehat{\mathbb{E}}_t\bigg(\int_t^T |\eta_s|^2ds\bigg)^{p/2}\leq
		\widehat{\mathbb{E}}_t\bigg[\sup_{u\in[t,T]}\bigg|\int_t^u\eta_s dB_s\bigg|^p\bigg]\leq
		\bar{\sigma}^p C\widehat{\mathbb{E}}_t\bigg(\int_t^T |\eta_s|^2ds\bigg)^{p/2},
	\end{displaymath}
	where $0<c<C<\infty$ are constants depending on $p, T$.
\end{proposition}

Let $$S_G^0(0,T):=\Big\{h(t,B_{t_1\wedge t}, \ldots,B_{t_n\wedge t}):t_1,\ldots,t_n\in[0,T],\; h\in C_{b,Lip}(\mathbb{R}^{n+1})\Big\}.$$ For $p\geq 1$ and $\eta\in S_G^0(0,T)$, set $$\|\eta\|_{S_G^p}:=\left\{\widehat{\mathbb{E}}\sup_{t\in[0,T]}|\eta_t|^p\right\}^{1/p}.$$ Denote by $S_G^p(0,T)$ the completion of $S_G^0(0,T)$ under the norm $\|\cdot\|_{S_G^p}$. We have the following uniform continuity property for the processes in $S_G^p(0,T)$. 

\begin{proposition}[\cite{li2018supermartingale}]\label{uniformcontinuous}
    For $Y\in S_G^p(0,T)$ with $p\geq 1$, we have, by setting $Y_s:=Y_T$ for $s>T$,
    \begin{align*}
        \limsup_{\varepsilon \rightarrow 0} \hE\left[\sup_{t\in[0,T]}\sup_{s\in[t,t+\varepsilon]}|Y_t-Y_s|^p\right]=0.
    \end{align*}
\end{proposition}

For $\xi\in L_{ip}(\Omega_T)$, let $$\mathcal{E}(\xi):=\widehat{\mathbb{E}}\left[\sup_{t\in[0,T]}\widehat{\mathbb{E}}_t[\xi]\right].$$ 
For $p\geq 1$ and $\xi\in L_{ip}(\Omega_T)$, define $$\|\xi\|_{p,\mathcal{E}}:=[\mathcal{E}(|\xi|^p)]^{1/p}$$ and denote by $L_{\mathcal{E}}^p(\Omega_T)$ the completion of $L_{ip}(\Omega_T)$ under $\|\cdot\|_{p,\mathcal{E}}$. The following theorem can be regarded as the Doob's maximal inequality under the $G$-expectation. 
\begin{theorem}[\cite{song2011some}]\label{the1.2}
	For any $\alpha\geq 1$ and $\delta>0$, $L_G^{\alpha+\delta}(\Omega_T)\subset L_{\mathcal{E}}^{\alpha}(\Omega_T)$. More precisely, for any $1<\gamma<\beta:=(\alpha+\delta)/\alpha$ and $\gamma\leq 2$, we have
	\begin{displaymath}
		\|\xi\|_{\alpha,\mathcal{E}}^{\alpha}\leq \gamma^*\Big\{\|\xi\|_{L_G^{\alpha+\delta}}^{\alpha}+14^{1/\gamma}
		C_{\beta/\gamma}\|\xi\|_{L_G^{\alpha+\delta}}^{(\alpha+\delta)/\gamma}\Big\},\qquad \forall \xi\in L_{ip}(\Omega_T),
	\end{displaymath}
	where $C_{\beta/\gamma}=\sum_{i=1}^\infty i^{-\beta/\gamma}$ and $\gamma^*=\gamma/(\gamma-1)$.
\end{theorem}
We can see that unlike the classical case, the order of the right-hand side is strictly greater than that of the left-hand side under the $G$-expectation.

\subsection{$G$-BSDEs}

We review some fundamental results about $G$-BSDEs. The solution of $G$-BSDE with terminal value $\xi$ and generators $f,g$, is a triple of processes $(Y,Z,A)$ evolve according to the following equation:
\begin{equation}\label{eqn:GBSDE}
	Y_t=\xi+\int_t^T f(s,Y_s,Z_s)ds+\int_t^T g(s,Y_s,Z_s)d\langle B\rangle_s-\int_t^T Z_s dB_s-(K_T-K_t),
\end{equation}
where $Y\in S_G^{\alpha}(0,T)$, $Z\in H_G^{\alpha}(0,T)$, and $K$ is a non-increasing $G$-martingale such that $K_0=0$ and $K_T\in L_G^{\alpha}(\Omega_T)$. To establish the well-posedness of the $G$-BSDE \eqref{eqn:GBSDE}, consider the generators 
\begin{displaymath}
	f(t,\omega,y,z), \  g(t,\omega,y,z):[0,T]\times\Omega_T\times\mathbb{R}\times\mathbb{R}\rightarrow \mathbb{R}
\end{displaymath}
satisfy the following properties: 
\begin{description}
	\item[(H1)] There exists some $\beta>1$, such that for any $y, z \in \mathbb{R}$,  $f(\cdot,\cdot,y,z), \ g(\cdot,\cdot,y,z)\in M_G^{\beta}(0,T)$;
	\item[(H2)] There exists some $\kappa>0$, such that
	\begin{displaymath}
		|f(t,\omega,y,z)-f(t,\omega,y',z')|+|g(t,\omega,y,z)-g(t,\omega,y',z')|\leq \kappa(|y-y'|+|z-z'|);
	\end{displaymath}
\item[(H3)] The terminal value $\xi\in L_G^{\beta}(\Omega_T)$. 
\end{description}

\begin{theorem}[\cite{hu2014backward}]\label{the1.4}
	Assuming that $f,g,\xi$ satisfy \textsc{(H1)}-\textsc{(H3)}, for any $1<\alpha<\beta$, the $G$-BSDE \eqref{eqn:GBSDE} has a unique solution $(Y,Z,K)$ satisfying that $Y\in S_G^{\alpha}(0,T)$, $Z\in H_G^{\alpha}(0,T)$, and $K$ is a non-increasing $G$-martingale such that $K_0=0$ and $K_T\in L_G^{\alpha}(\Omega_T)$. Moreover, 
	\begin{displaymath}
		|Y_t|^\alpha\leq C\widehat{\mathbb{E}}_t\left[|\xi|^\alpha+\int_t^T \big(|f(s,0,0)|^\alpha+|g(s,0,0)|^\alpha \big)ds\right],
	\end{displaymath}
	where the constant $C$ depends on $\alpha$, $T$, $\underline{\sigma}$ and $\kappa$.
\end{theorem}

The following results will be needed in our proofs. Note that $(Y,Z,K)$ in Theorem \ref{Esti-Z-GBSDE} is not the solution to the $G$-BSDE \eqref{eqn:GBSDE}.
\begin{theorem}[\cite{li2021backward}] \label{Esti-Z-GBSDE}
	Let $f,g$ satisfy (H1) and (H2) for some $\beta>1$. Assume
	\[
	Y_{t}=\xi+\int_{t}^{T}f(s,Y_{s},Z_{s})ds+\int_t^T g(s,Y_s,Z_s)d\langle B\rangle_s-\int_{t}^{T}Z_{s}dB_{s}-(K_{T}
	-K_{t})+(A_T-A_t),
	\]
	where $Y\in S_G^{\alpha}(0,T)$, $Z\in H_G^{\alpha}(0,T)$, and both $K$ and $A$ are
	non-increasing processes such that $A_0=K_{0}=0$ and $A_T, K_{T}\in L_G^{\alpha}(\Omega
	_{T})$ for some $1<\alpha<\beta$. Then there exists a constant $C$ that depends on $\alpha$, $T$, $\underline{\sigma}$ and $\kappa$, such that
	\begin{align*}
		&\mathbb{\widehat{E}}\left(\int_{0}^{T}|Z_{s}|^{2}ds\right)^{\frac{\alpha}{2}}\\
		&\leq
		C\Bigg\{ \mathbb{\widehat{E}}|Y^*_{T}|^{\alpha}+\big(\mathbb{\widehat{E}}|Y^*_{T}|^{\alpha}\big)^{\frac{1}{2}%
		}\bigg[\bigg(\mathbb{\widehat{E}}\bigg(\int_{0}^{T}|f(s,0,0)| ds\bigg)^{\alpha}\bigg)^{\frac{1}{2}}\\
	&\hspace{5cm}+\bigg(\mathbb{\widehat{E}}\bigg(\int_{0}^{T}|g(s,0,0) |ds\bigg)^{\alpha}\bigg)^{\frac{1}{2}}+\big(m_\alpha^{A, K}\big)^{1/2}\bigg]\Bigg\},
	\end{align*}
	where $Y^*_T=\sup_{t\in[0,T]}|Y_t|$ and $m_\alpha^{A,K}=\min \Big\{\widehat{\mathbb{E}}|A_T|^\alpha,\, \widehat{\mathbb{E}}|K_T|^\alpha\Big\}$.
\end{theorem}

Similar to the classical case, the comparison theorem for $G$-BSDEs still holds.
\begin{theorem}[\cite{hu2014comparison}]\label{the1.5}
	For $l=1,2$, let $(Y_t^l,Z_t^l,K_t^l)_{t\leq T}$  be the solution of the following $G$-BSDE:
	\begin{displaymath}
		Y^l_t=\xi^l+\int_t^T f^l(s,Y^l_s,Z^l_s)ds+\int_t^T g^l(s,Y^l_s,Z^l_s)d\langle B\rangle_s+V_T^l-V_t^l-\int_t^T Z^l_s dB_s-(K^l_T-K^l_t),
	\end{displaymath}
	where processes $\{V_t^l\}_{0\leq t\leq T}$ are assumed to be right-continuous with left limits q.s., such that $\widehat{\mathbb{E}}[\sup_{t\in[0,T]}|V_t^l|^\beta]<\infty$.  Assuming that $\xi^l$, $f^l,\ g^l$ satisfy \textsc{(H1)}-\textsc{(H3)} for $l=1,2$, if $\xi^1\geq \xi^2$, $f^1\geq f^2$, $g^1\geq g^2$, and $V^1-V^2$ is a non-decreasing process, then $Y_t^1\geq Y_t^2$.
\end{theorem}

In contrast to classical BSDEs, the  inclusion of the additional non-increasing $G$-martingale $K$ in $G$-BSDEs  introduces model uncertainty and complicates the analysis.   
\cite{song2019properties} demonstrated that the non-increasing $G$-martingale cannot be expressed in the form ${\int_0^t \eta_sdt}$ or ${\int_0^t \gamma_sd\langle B\rangle_s}$, where $\eta,\gamma\in M_G^1(0,T)$. Specifically, the author established the following result.
\begin{theorem}[\cite{song2019properties}]\label{the1.6}
	Assume that for $t\in[0,T]$, $$\int_0^t \zeta_s dB_s+\int_0^t \eta_sds+K_t=L_t,$$ where $\zeta\in H_G^1(0,T)$, $\eta\in M_G^1(0,T)$, and $K,L$ are non-increasing $G$-martingales. Then we have $$\int_0^t \zeta_sdB_s=0,\quad \int_0^t \eta_sds=0\quad\text{and} \quad K_t=L_t.$$
\end{theorem}
	We call the following process $u$ a generalized $G$-It\^{o} process:
	\begin{displaymath}
		u_t=u_0+\int_0^t \eta_sds+\int_0^t \zeta_s dB_s+K_t,
	\end{displaymath}
	where $\eta\in M_G^1(0,T)$, $\zeta\in H_G^1(0,T)$, and $K$ is a non-increasing $G$-martingale such that $K_0=0$. By Theorem \ref{the1.6},  the decomposition of the generalized $G$-It\^{o} process is unique.

\subsection{Reflected $G$-BSDEs with a single obstacle}

Now we review the reflected $G$-BSDE with a lower obstacle studied in \cite{li2018reflected}. Their parameters consist of a terminal value $\xi$, generators $f,g$, and an obstacle $L$, where $L$ satisfies the following condition:
\begin{description}
	\item[(H4)] $L\in S_G^\beta(0,T)$ is bounded from above by a generalized $G$-It\^{o} process $L'$ of the following form:
	\begin{displaymath}
		L'_t=L'_0+\int_0^t b'(s)ds+\int_0^t \sigma'(s)dB_s+K'_t,
	\end{displaymath}
	where $b'\in M_G^\beta(0,T)$, $\sigma'\in H_G^\beta(0,T)$, and $K'\in S_G^\beta(0,T)$ is a non-increasing $G$-martingale such that $K'_0=0$ and $\beta>2$. Additionally, $\xi\geq L_T$ q.s.
\end{description}
 A triple of processes $(Y,Z,A)$ for some $1<\alpha< \beta$, is called a solution of the reflected $G$-BSDE with a lower obstacle with parameters $(\xi,f,g,L)$, if
\begin{align}
	\label{eqn:RGBSDE}
	\begin{cases}
		Y_t=\xi+\int_t^T f(s,Y_s,Z_s)ds+\int_t^T g(s,Y_s,Z_s)d\langle B\rangle_s
		-\int_t^T Z_s dB_s+(A_T-A_t),\vspace{0.2cm}\\
		Y_t\geq L_t,\quad 0\leq t\leq T, \vspace{0.2cm}\\
		\big\{-\int_0^t (Y_s-L_s)dA_s\big\}_{t\in[0,T]} \textrm{ is a non-increasing $G$-martingale},
	\end{cases}
\end{align}
where $Y\in S_G^{\alpha}(0,T)$, $Z\in H_G^{\alpha}(0,T)$, and $A$ is a continuous non-decreasing process such that $A_0=0$ and $A\in S_G^\alpha(0,T)$. The following theorem provides the well-posedness  of the reflected $G$-BSDE \eqref{eqn:RGBSDE}.

\begin{theorem}[\cite{li2018reflected}]\label{the1.15}
	Suppose that $\xi$, $f$, $g$ and $L$ satisfy \textsc{(H1)}--\textsc{(H4)} with $\beta>2$. Then the reflected $G$-BSDE \eqref{eqn:RGBSDE} has a unique solution $(Y,Z,A)$. Moreover, for any $2\leq \alpha<\beta$ we have $Y\in S^\alpha_G(0,T)$, $Z\in H_G^\alpha(0,T)$ and $A\in S_G^{\alpha}(0,T)$.
\end{theorem}

The following theorem provides the comparison theorem  for the reflected $G$-BSDE \eqref{eqn:RGBSDE}.
\begin{theorem}[\cite{li2018reflected}]\label{the1.16}
Suppose $\xi^i$, $L^i$, $f^i$ and $g^i$ for $i=1,2$ satisfy \textsc{(H1)}--\textsc{(H4)} with $\beta>2$. Furthermore, assume the following:
	\begin{enumerate}[label=(\roman*)]
		\item $\xi^1\leq \xi^2$, $q.s.$;
		\item $f^1(t,y,z)\leq f^2(t,y,z)$ and $g^1(t,y,z)\leq g^2(t,y,z)$, $\forall (y,z)\in\mathbb{R}^2$;
		\item $L_t^1\leq L^2_t$, $0\leq t\leq T$, q.s..
	\end{enumerate}
	Let $(Y^i,Z^i,A^i)$ be the solution of the reflected $G$-BSDE \eqref{eqn:RGBSDE} with parameters $(\xi^i,f^i,g^i,L^i)$ for $i=1,2$. Then
$Y_t^1\leq Y^2_t$ for $0\leq t\leq T$ q.s.

\end{theorem}

\section{Well-posedness of doubly reflected $G$-BSDEs}
\label{sec:Well-posedness}
In this section, we consider doubly reflected $G$-BSDEs and establish their well-posedness. Specifically, in Subsection \ref{sec:DRGBSDE}, we first define their solutions and present our first main result in Theorem \ref{main}; in Subsection \ref{sec:Preliminary_analysis}, we conduct preliminary analysis; Subsection \ref{sec:mainproof} is dedicated to the Proof of Theorem \ref{main}.

\subsection{Doubly reflected $G$-BSDEs}
\label{sec:DRGBSDE}
A triple of processes $(Y,Z,A)$, with $Y, A\in S_G^\alpha(0,T)$ and $Z\in H_G^\alpha(0,T)$ for some $2\leq\alpha<\beta$, is called a solution to the doubly reflected $G$-BSDE with the parameters $(\xi, f, g, L, U)$,  if 
\begin{align}
	\label{eqn:DRGBSDE}
	\begin{cases}
		Y_t=\xi+\int_t^T f(s,Y_s,Z_s)ds+\int_t^T g(s,Y_s,Z_s)d\langle B\rangle_s-\int_t^T Z_s dB_s+(A_T-A_t),\vspace{0.2cm}\\
		L_t\leq Y_t\leq U_t,\quad 0\leq t\leq T, \vspace{0.2cm}\\
		(Y, A) \textrm{ satisfies the } \textmd{ASC}_{\alpha}.
	\end{cases}
\end{align}
A pair of processes $(Y,  A)$ with $Y, A\in S_G^\alpha(0,T)$ is said to satisfy the  $\textmd{ASC}_{\alpha}$, if  there exist non-decreasing processes $\{A^{n,+}\}_{n\in\mathbb{N}}$, $\{A^{n,-}\}_{n\in\mathbb{N}}$, and  non-increasing $G$-martingales $\{K^n\}_{n\in\mathbb{N}}$, such that
\begin{itemize}
	\item $\widehat{\mathbb{E}}\Big[|A_T^{n,+}|^\alpha+|A_T^{n,-}|^\alpha+|K^n_T|^\alpha\Big]\leq C$, where $C$ is  independent of $n$;
	\item $\widehat{\mathbb{E}}\sup\limits_{t\in[0,T]}\Big|A_t-(A_t^{n,+}-A_t^{n,-}-K_t^n)\Big|^\alpha\rightarrow 0$, as $n\rightarrow\infty$;
	\item  $\lim\limits_{n\rightarrow\infty}\widehat{\mathbb{E}}\left|\int_0^T (Y_s-L_s)d A_s^{n,+}\right|^{\alpha/2}=0$;
	\item  $\lim\limits_{n\rightarrow\infty}\widehat{\mathbb{E}}\left|\int_0^T (U_s-Y_s)d A_s^{n,-}\right|^{\alpha/2}=0$.
\end{itemize}
We call $\{A^{n,+}\}_{n\in\mathbb{N}}$, $\{A^{n,-}\}_{n\in\mathbb{N}}$, and $\{K^n\}_{n\in\mathbb{N}}$ the approximate sequences for $(Y,A)$ with order $\alpha$ w.r.t. the lower obstacle $L$ and the upper obstacle $U$. 

Consider the parameters of the doubly reflected $G$-BSDE \eqref{eqn:DRGBSDE}, namely the terminal value $\xi$, the generators $f,g$, and the obstacles $L,U$, satisfy (H2), (H3) and the following assumptions:
\begin{description}
	\item[(A1)] There exists some $\beta>2$, such that for any $y,z\in \mathbb{R}$, $f(\cdot,\cdot,y,z)$, $g(\cdot,\cdot,y,z)\in S_G^\beta(0,T)$;
	\item[(A2)] $L$, $U\in S_G^\beta(0,T)$. There exists some $I\in S_G^\beta(0,T)$ satisfying the following representation:
	\begin{displaymath}
		I_t=I_0+A^{I,-}_t-A_t^{I,+}+\int_0^t \sigma^I(s)dB_s.
	\end{displaymath}
	Here, $A^{I,+}$, $A^{I,-}\in S_G^\beta(0,T)$ are two non-decreasing processes such that $A^{I,+}_0=A^{I,-}_0=0$; $\sigma^I\in S_G^\beta(0,T)$ satisfies $L_t\leq I_t\leq U_t$; $U+A^{I,+}$ is a generalized $G$-It\^{o} process evolves according to
	\begin{equation}\label{eqn:A2}
		U_t+A^{I,+}_t=U_0+\int_0^t b(s)ds+\int_0^t \sigma(s)dB_s+K^u_t,
	\end{equation}
	where $b,\sigma\in S_G^\beta(0,T)$, and $K^u\in S_G^\beta(0,T)$ is a non-increasing $G$-martingale such that $K^u_0=0$. Additionally, $L_T\leq \xi\leq U_T$, q.s..
\end{description}

\begin{remark}
In comparison to \cite{li2021backward}, their conditions are the same as ours except Assumption (A3) therein which corresponds to our Assumption (A2), while ours is weaker. Specifically, their (A3) says that the upper obstacle is a generalized $G$-It\^{o} process of the following form:
	\begin{displaymath}
		U_t=U_0+\int_0^t b^U(s)ds+\int_0^t \sigma^U(s)dB_s+K^U_t,
	\end{displaymath}
	where $b^U,\sigma^U\in S_G^\beta(0,T)$, and $K^U\in S_G^\beta(0,T)$ is a non-increasing $G$-martingale. Setting $$I=U,\qquad \sigma^I=\sigma^U,\qquad A^{I,-}_t=\int_0^t (b^U(s))^+ds,\quad\text{and}\quad A^{I,+}_t=\int_0^t (b^U(s))^-ds-K^U_t,$$
 their $(L,U)$ pair clearly satisfies (A2) of this paper.
\end{remark}

Theorem \ref{main} below is our first main result. It firstly estalishes the well-posedness of the doubly reflected $G$-BSDE \eqref{eqn:DRGBSDE} using the weakest known regularity conditions. Secondly, it establishes that the first component of the solution to \eqref{eqn:DRGBSDE} can be approximated by a monotone sequence of processes, which are the solutions to a family of penalized single reflected $G$-BSDEs. This construction will play a fundamental role to establishing the connection between doubly reflected $G$-BSDEs and fully nonlinear PDEs with double obstacles. The proof of Theorem \ref{main} is provided in Subsection \ref{sec:mainproof}.
\begin{theorem}\label{main}
	Assuming that $\xi$, $f$, $g$, $L$ and $U$ satisfy Assumptions (H2)-(H3) and (A1)-(A2), the following properties hold  for any $2\leq \alpha<\beta$:
	\begin{enumerate}[label=(\alph*)]
\item The doubly reflected $G$-BSDE \eqref{eqn:DRGBSDE} has a unique solution $(Y,Z,A)$, such that $Y\in S^\alpha_G(0,T)$, $Z\in H_G^\alpha(0,T)$ and $A\in S_G^{\alpha}(0,T)$.

\item This $Y$ can be approximated by a monotone decreasing sequence of processes $\bar{Y}^n$ (i.e. $\bar{Y}^{n_1}_t\geq \bar{Y}^{n_2}_t$  for any $n_1\leq n_2$) in the sense that 
	\begin{align*}
	\lim_{n\rightarrow \infty}\widehat{\mathbb{E}}\sup_{t\in[0,T]}|Y_t-\bar{Y}^n_t|^\alpha=0, 
\end{align*}
	where $(\bar{Y}^n,\bar{Z}^n,\bar{A}^n)$ for each $n\in\mathbb{N}$ is the solution to the following reflected $G$-BSDE:
\begin{align}
	\label{barY^n}
	\begin{cases}
		&\bar{Y}^n_t=\xi+\int_t^T f(s,\bar{Y}^n_s,\bar{Z}^n_s)ds+\int_t^T g(s,\bar{Y}^n_s,\bar{Z}^n_s)d\langle B\rangle_s-n\int_t^T(\bar{Y}^n_s-U_s)^+ds\\
		&\hspace{1cm}-\int_t^T \bar{Z}^n_sdB_s+(\bar{A}^n_T-\bar{A}^n_t),\vspace{0.2cm}\\
		&\bar{Y}^n_t\geq L_t,\quad 0\leq t\leq T, \vspace{0.2cm}\\
		&\big\{-\int_0^t (\bar{Y}^n_s-L_s)d\bar{A}^n_s\big\}_{t\in[0,T]} \textrm{ is a non-increasing $G$-martingale}.
	\end{cases}
\end{align}

\item The remaining terms $(Z,A)$ can be constructed by the penalized reflected $G$-BSDEs \eqref{barY^n}, in the way that
\begin{displaymath}
	\lim_{n\rightarrow\infty}\widehat{\mathbb{E}}\left(\int_0^T |\bar{Z}^n_s-Z_s|^2ds\right)^{\frac{\alpha}{2}}=0\quad\text{and}\quad \lim_{n\rightarrow\infty}\widehat{\mathbb{E}}\sup_{t\in[0,T]}|\widetilde{A}^n_t-A_t|^\alpha=0,
\end{displaymath}
where $\widetilde{A}^n_t=\bar{A}^n_t-\int_0^t n(\bar{Y}^n_s-U_s)^+ds$. 
	\end{enumerate}

\end{theorem}

\subsection{Preliminary analysis}
\label{sec:Preliminary_analysis}

In this subsection, we conduct preliminary analysis in order to prove Theorem \ref{main}.  Firstly, we aim to establish the uniform boundedness of $\bar{Y}^n$ under the norm $\|\cdot\|_{S_G^\alpha}$. Note that by Theorem \ref{the1.15}, the reflected $G$-BSDE \eqref{barY^n} admits a unique solution $(\bar{Y}^n,\bar{Z}^n,\bar{A}^n)$ for any $n\in\mathbb{N}$, satisfying $Y\in S_G^{\alpha}(0,T)$ and $Z\in H_G^{\alpha}(0,T)$ for $1<\alpha<\beta$, and $K$ is a non-increasing $G$-martingale such that $K_0=0$ and $K_T\in L_G^{\alpha}(\Omega_T)$. Then, we demonstrate that $(\bar{Y}^n-U)^+$ converges to $0$ with an explicit rate of $\frac{1}{n}$ and subsequently derive uniform estimates for $\bar{Z}^n$ and $\bar{A}^n$ under the norms $\|\cdot\|_{H_G^\alpha}$ and $\|\cdot\|_{L_G^\alpha}$, respectively.
However, given that $\bar{A}^n$ is not a $G$-martingale, we encounter some difficulties. To address this challenge, for each fixed $n\in\mathbb{N}$, we approximate the solution to \eqref{barY^n} by the solutions to the following family of $G$-BSDEs parameterized by $m\in\mathbb{N}$:
\begin{equation}\label{Y^{n,m}}\begin{split}
		Y^{n,m}_t=&\xi+\int_t^T f(s,Y^{n,m}_s,Z^{n,m}_s)ds+\int_t^T g(s,Y^{n,m}_s,Z^{n,m}_s)d\langle B\rangle_s-\int_t^T Z_s^{n,m}dB_s\\
		&+\int_t^T m(Y_s^{n,m}-L_s)^-ds-\int_t^T n(Y_s^{n,m}-U_s)^+ds-(K_T^{n,m}-K_t^{n,m}).
\end{split}\end{equation}
Set 
\begin{align}\label{anm+-}
A^{n,m,+}_t=\int_0^t m(Y_s^{n,m}-L_s)^-ds \quad\text{and}\quad A^{n,m,-}_t=\int_0^t n(Y_s^{n,m}-U_s)^+ds.
\end{align}
Clearly, $A^{n,m,+}$ and $A^{n,m,-}$ are non-decreasing processes and Equation \eqref{Y^{n,m}} can be rewritten as:
\begin{equation}
	\begin{split}
		Y^{n,m}_t=&\xi+\int_t^T f(s,Y^{n,m}_s,Z^{n,m}_s)ds+\int_t^T g(s,Y^{n,m}_s,Z^{n,m}_s)d\langle B\rangle_s-\int_t^T Z_s^{n,m}dB_s\\ &+(A^{n,m,+}_T-A^{n,m,+}_t)-(A^{n,m,-}_T-A^{n,m,-}_t)-(K_T^{n,m}-K_t^{n,m}).
\end{split}\end{equation}

In the following, we show that  under Assumptions (H2)-(H3) and (A1)-(A2), the sequence $(Y^{n,m},Z^{n,m}, A^{n,m,+}-K^{n,m})$ converges to $(\bar{Y}^n,\bar{Z}^n,\bar{A}^n)$ as $m$ goes to infinity. The initial step involves establishing the uniform boundedness of $Y^{n,m}$  under the norm $\|\cdot\|_{S_G^\alpha}$. Note that since the upper obstacle here is no longer a generalized $G$-It\^{o} process, conventional approaches found in existing literature on reflected G-BSDEs are inapplicable. 
Our technical proofs commence with the following lemma, wherein we utilize a weak condition that is fulfilled by the conditions presented in subsequent proofs.
\begin{lemma}\label{est-Y}
Assuming that $\xi$, $f$, $g$, $L$ and $U$ satisfy Assumptions (H1)-(H3) and the (A2') below (which is essentially (A2) but without the requirements on $U+A^{I,+}$):
\begin{description}
	\item[(A2')] $L$, $U\in S_G^\beta(0,T)$. There exists some $I\in S_G^\beta(0,T)$ satisfying the following representation
	\begin{displaymath}
		I_t=I_0+A^{I,-}_t-A_t^{I,+}+\int_0^t \sigma^I(s)dB_s,
	\end{displaymath}
	where $A^{I,+}$, $A^{I,-}\in S_G^\beta(0,T)$ are two non-decreasing processes with $A^{I,+}_0=A^{I,-}_0=0$ and $\sigma^I\in S_G^\beta(0,T)$ such that $L_t\leq I_t\leq U_t$. Additionally, $L_T\leq \xi\leq U_T$, q.s.
\end{description}
Then there exists a constant $C$ independent of $n,m$, such that for $2\leq \alpha<\beta$, 
	\begin{displaymath}
		\widehat{\mathbb{E}}\sup_{t\in[0,T]}|Y_t^{n,m}|^\alpha\leq C.
	\end{displaymath}
\end{lemma}

\begin{proof}
	Set $Y^*_t=I_t$ and $Z^*_t=\sigma^I_t$. It is easy to check that
	\begin{equation}\label{Y^*}\begin{split}
			Y^*_t=&I_T-\int_t^T Z^*_sd B_s+(A^{I,+}_T-A^{I,+}_t)-(A^{I,-}_T-A^{I,-}_t)\\
			=&I_T+\int_t^T f(s,Y^*_s,Z^*_s)ds+\int_t^T g(s,Y^*_s,Z^*_s)d\langle B\rangle_s\\
            &-\int_t^T Z^*_s dB_s+(A^{*,+}_T-A^{*,+}_t)-(A^{*,-}_T-A^{*,-}_t),
	\end{split}\end{equation}
	where 
 \begin{align}
     &A^{*,+}_t=A^{I,+}_t+\int_0^t f^-(s,Y^*_s,Z^*_s)ds+\int_0^t g^-(s,Y^*_s,Z^*_s)d\langle B\rangle_s,\label{eqn:Astarplus}\\
     & A^{*,-}_t=A^{I,-}_t+\int_0^t f^+(s,Y^*_s,Z^*_s)ds+\int_0^t g^+(s,Y^*_s,Z^*_s)d\langle B\rangle_s.\label{eqn:Astarmius}
 \end{align} 
 Clearly, $A^{*,+}, A^{*,-}\in S_G^\beta(0,T)$ are non-decreasing processes. Consider the following two $G$-BSDEs:
	\begin{equation}\label{Y^+}\begin{split}
	    Y_t^+=&U_T+\int_t^T f(s,Y_s^+,Z_s^+)ds+\int_t^T g(s,Y_s^+,Z_s^+)d\langle B\rangle_s+(A^{*,+}_T-A^{*,+}_t)\\
     &-\int_t^T Z_s^+ dB_s-(K^+_T-K^+_t),
	\end{split}
	\end{equation}
    \begin{equation}\label{Y^-}\begin{split}
		 Y_t^-=&L_T+\int_t^T f(s,Y_s^-,Z_s^-)ds+\int_t^T g(s,Y_s^-,Z_s^-)d\langle B\rangle_s-(A^{*,-}_T-A^{*,-}_t)\\
   &-\int_t^T Z_s^- dB_s-(K^-_T-K^-_t).
	\end{split}\end{equation}
	By Theorem \ref{the1.5}, we have $Y_t^-\leq Y_t^*\leq Y_t^+$ for any $t\in[0,T]$, which implies that $$Y_t^+\geq I_t\geq L_t\qquad \text{and} \qquad Y_t^-\leq I_t\leq U_t.$$ Therefore, we may add the terms $+\int_t^T m(Y_s^+-L_s)^-ds$ and $-\int_t^T n(Y_s^--U_s)^+ds$ to Equations \eqref{Y^+} and \eqref{Y^-}, respectively. By Theorem \ref{the1.5} again, we have $Y_t^-\leq Y^{n,m}_t\leq Y^+_t$ for any $t\in[0,T]$ and $n,m\in\mathbb{N}$. By the estimates for $G$-BSDEs (see Theorem \ref{the1.4}), we have
	\begin{align*}
		&\hspace{-0.7cm}\widehat{\mathbb{E}}\sup_{t\in[0,T]}|Y_t^++A_t^{*,+}|^\alpha\\
		&\leq C\widehat{\mathbb{E}}\left[\sup_{t\in[0,T]}\widehat{\mathbb{E}}_t\left[|U_T+A_T^{*,+}|^\alpha+\int_t^T \Big(|f(s,0,0)|^\alpha+|g(s,0,0)|^\alpha+|A_s^{*,+}|^\alpha\Big)ds\right]\right],\\
		&\hspace{-0.7cm}\widehat{\mathbb{E}}\sup_{t\in[0,T]}|Y_t^--A_t^{*,-}|^\alpha\\
		&\leq C\widehat{\mathbb{E}}\left[\sup_{t\in[0,T]}\widehat{\mathbb{E}}_t\left[|L_T-A_T^{*,-}|^\alpha+\int_t^T \Big(|f(s,0,0)|^\alpha+|g(s,0,0)|^\alpha+|A_s^{*,-}|^\alpha\Big)ds\right]\right].
	\end{align*}
By Theorem \ref{the1.2} and H\"{o}lder's inequality, there exists a constant $C$ independent of $n,m$ such that 
   $$\widehat{\mathbb{E}}\sup_{t\in[0,T]}|Y_t^+|^\alpha\leq C \qquad\text{and}\qquad
   \widehat{\mathbb{E}}\sup_{t\in[0,T]}|Y_t^-|^\alpha\leq C.$$ Consequently, we have
   $$\widehat{\mathbb{E}}\sup_{t\in[0,T]}|Y_t^{n,m}|^\alpha\leq C,$$ 
   where $C$ is a constant independent of $n,m$.
\end{proof}

The following lemma provides the explicit convergence rate of $(Y^{n,m}-U)^+$, which will be instrumental in deriving the convergence rate of $(\bar{Y}^n-U)^+$. The latter is challenging to obtain solely by considering the penalization sequence \eqref{barY^n}, as $\bar{A}^n$ does not exhibit the properties of a non-increasing $G$-martingale. To address this limitation, we introduce the penalization sequence with two parameters $n$ and $m$ in \eqref{Y^{n,m}}. Although Equation \eqref{eqn:A2} is not required in Lemma \ref{est-Y}, it becomes necessary starting from this point onward.
\begin{lemma}\label{Y^{n,m}-U}
Assuming that $\xi$, $f$, $g$, $L$ and $U$ satisfy Assumptions (H2)-(H3) and (A1)-(A2).	There exists a constant $C$ independent of $n,m$, such that for $2\leq \alpha<\beta$, 
	\begin{displaymath}
		\widehat{\mathbb{E}}\sup_{t\in[0,T]}\big|(Y_t^{n,m}-U_t)^+\big|^\alpha\leq \frac{C}{n^\alpha}.
	\end{displaymath}
\end{lemma}

\begin{proof}
	Consider the following $G$-BSDE:
	\begin{equation}\begin{split}\label{hatY^n}
			\widehat{Y}^n_t=&U_T+\int_t^T f(s,\widehat{Y}^n_s,\widehat{Z}^n_s)ds+\int_t^T g(s,\widehat{Y}^n_s,\widehat{Z}^n_s)d\langle B\rangle_s-\int_t^T \widehat{Z}^n_sd B_s\\ &-\int_t^T n(\widehat{Y}^n_s-U_s)^+ds+(A_T^{*,+}-A_t^{*,+})-(\widehat{K}^n_T-\widehat{K}^n_t),
	\end{split}\end{equation}
	where $A^{*,+}$ is defined in \eqref{eqn:Astarplus}. By Equation \eqref{Y^+} and Theorem \ref{the1.5}, we have $\widehat{Y}^n_t\leq Y^+_t$ for $n\in\mathbb{N}$. Noting that $Y^*_t=I_t\leq U_t$, we may add $-\int_t^T n(Y^*_s-U_s)^+ds$ to Equation \eqref{Y^*}. By Theorem \ref{the1.5}, we have $\widehat{Y}^n_t\geq Y^*_t$ and hence $\widehat{Y}^n_t\geq L_t$ for any $n\in\mathbb{N}$ and $t\in[0,T]$. Therefore, we may add $+\int_t^T m(\widehat{Y}^n_s-L_s)^-ds$ to Equation \eqref{hatY^n}. Applying Theorem \ref{the1.5} again yields $\widehat{Y}^n_t\geq Y^{n,m}_t$.  It suffices to prove that  there exists a constant $C$ independent of $n,m$, such that for any $2\leq \alpha<\beta$,
	\begin{displaymath}
		\widehat{\mathbb{E}}\sup_{t\in[0,T]}\big|(\widehat{Y}^n_t-U_t)^+\big|^\alpha\leq \frac{C}{n^\alpha}.
	\end{displaymath}
	Set $$\widetilde{Y}^n_t=\widehat{Y}^n_t+A^{*,+}_t,\quad \widetilde{\xi}=U_T+A_T^{*,+}\quad \text{and}\quad\widetilde{U}_t=U_t+A^{*,+}_t.$$ Equation \eqref{hatY^n} can be rewritten as
	\begin{align*}
		\widetilde{Y}^n_t=\widetilde{\xi}+\int_t^T \widetilde{f}(s,\widetilde{Y}^n_s,\widehat{Z}^n_s)ds+\int_t^T \widetilde{g}(s,\widetilde{Y}^n_s,\widehat{Z}^n_s)d\langle B\rangle_s-&\int_t^T n(\widetilde{Y}^n-\widetilde{U}_s)^+ds-\int_t^T \widehat{Z}^n_sdB_s\\&\hspace{2cm}-(\widehat{K}^n_T-\widehat{K}_t^n),
	\end{align*}
	where $$\widetilde{f}(s,y,z)=f(s,y-A^{*,+}_s,z)\qquad \text{and}\qquad\widetilde{g}(s,y,z)=g(s,y-A^{*,+}_s,z).$$ Given that $Y^*_t\leq \widehat{Y}^n_t\leq Y^+_t$ for any $n\in\mathbb{N}$, there exists a constant $C$ independent of $n$, such that $$\widehat{\mathbb{E}}\sup_{t\in[0,T]}|\widehat{Y}_t^{n}|^\alpha\leq C.$$
 Consequently, 
 $$\widehat{\mathbb{E}}\sup_{t\in[0,T]}|\widetilde{Y}_t^{n}|^\alpha\leq C,$$
 where $C$ is independent of $n$. 
 By Lemma 4.5 in \cite{li2017reflected}, we have
	\begin{displaymath}
		\widehat{\mathbb{E}}\sup_{t\in[0,T]}\big|(\widetilde{Y}^n_t-\widetilde{U}_t)^+\big|^\alpha\leq \frac{C}{n^\alpha},
	\end{displaymath}
	which yields the desired result.
\end{proof}

Next, we show that the sequences $A^{n,m,+}$, $A^{n,m,-}$, $K^{n,m}$ and $Z^{n,m}$ are uniformly bounded.
\begin{lemma}\label{A^{n,m,+},Z^{n,m},K^{n,m}}
Assuming that $\xi$, $f$, $g$, $L$, and $U$ satisfy Assumptions (H2)-(H3) and (A1)-(A2).	There exists a constant $C$ independent of $n,m$, such that for $2\leq \alpha<\beta$, 
	\begin{displaymath}
		\widehat{\mathbb{E}}|A_T^{n,m,+}|^\alpha\leq C,\quad	\widehat{\mathbb{E}}|A_T^{n,m,-}|^\alpha\leq C,\quad	\widehat{\mathbb{E}}|K_T^{n,m}|^\alpha\leq C\quad \text{and}\quad	\widehat{\mathbb{E}}\left(\int_0^T|Z_s^{n,m}|^2ds\right)^{\alpha/2}\leq C.
	\end{displaymath}
\end{lemma}

\begin{proof}
	By Lemma \ref{Y^{n,m}-U} and the definition of $A^{n,m,-}$ given in Equation \eqref{anm+-}, it is easy to check that $\widehat{\mathbb{E}}|A_T^{n,m,-}|^\alpha\leq C$. We have by Theorem \ref{Esti-Z-GBSDE} that
	\begin{align*}
		&\widehat{\mathbb{E}}\left(\int_0^T|Z_s^{n,m}|^2ds\right)^{\alpha/2}\\
		&\leq C\Bigg\{\widehat{\mathbb{E}}\sup_{t\in[0,T]}|Y^{n,m}_t|^\alpha+\bigg(\widehat{\mathbb{E}}\sup_{t\in[0,T]}|Y^{n,m}_t|^\alpha\bigg)^{1/2}\\
		&\hspace{1cm}\times\Bigg(\Bigg(\widehat{\mathbb{E}}\int_0^T |f(s,0,0)|^\alpha ds\Bigg)^{1/2}+\Bigg(\widehat{\mathbb{E}}\int_0^T |g(s,0,0)|^\alpha ds\Bigg)^{1/2}+\Big(\widehat{\mathbb{E}}|A_T^{n,m,-}|^\alpha\Big)^{1/2}\Bigg)\Bigg\}.
	\end{align*}
Noting that (H1) is weaker than (A1) with $\beta>2$, we obtain by Lemma \ref{est-Y} that $$\widehat{\mathbb{E}}\left(\int_0^T|Z_s^{n,m}|^2ds\right)^{\alpha/2}\leq C.$$
Further note that
	\begin{align*}
		A^{n,m,+}_T-K^{n,m}_T=&Y_0^{n,m}-\xi+\int_0^T Z_s^{n,m}dB_s+A^{n,m,-}_T\\
        &-\int_0^T f(s,Y^{n,m}_s,Z^{n,m}_s)ds-\int_0^T g(s,Y^{n,m}_s,Z^{n,m}_s)d\langle B\rangle_s.
	\end{align*}
	By simple calculation, we obtain that
	\begin{align*}
		\widehat{\mathbb{E}}|A^{n,m,+}_T-K^{n,m}_T|^\alpha\leq & C\Bigg\{\widehat{\mathbb{E}}\sup_{t\in[0,T]}|Y^{n,m}_t|^\alpha+\widehat{\mathbb{E}}\bigg(\int_0^T|Z_s^{n,m}|^2ds\bigg)^{\alpha/2}+\widehat{\mathbb{E}}|A_T^{n,m,-}|^\alpha\\
		&\hspace{2.5cm}+\widehat{\mathbb{E}}\int_0^T |f(s,0,0)|^\alpha ds+\widehat{\mathbb{E}}\int_0^T |g(s,0,0)|^\alpha ds\Bigg\}.
	\end{align*}
	Since $A^{n,m,+}_T$ and $-K^{n,m}_T$ are non-negative, we obtain the desired result.
\end{proof}


By a similar analysis as the proof of Lemma 4.3, Lemma 4.4 and Theorem 5.1 in \cite{li2018reflected}, we have for any fixed $n$ and $2\leq \alpha<\beta$, 
\begin{align}
	\label{statementA}
	\lim_{m\rightarrow\infty}\widehat{\mathbb{E}}\left[\sup_{t\in[0,T]}|(Y^{n,m}_t-L_t)^-|^\alpha\right]=0,
\end{align}
 and letting $m$ go to infinity, $(Y^{n,m},Z^{n,m},A^{n,m,+}-K^{n,m})$ converges to $(\bar{Y}^n,\bar{Z}^n,\bar{A}^n)$, which is the solution of Equation \eqref{barY^n}. Specifically, we have
	\begin{equation}
	\begin{split}
		\label{statementB}
		&\lim_{m\rightarrow\infty}\widehat{\mathbb{E}}\left[\sup_{t\in[0,T]}|\bar{Y}^n_t-Y^{n,m}_t|^\alpha\right]=0,\qquad\lim_{m\rightarrow\infty}\widehat{\mathbb{E}}\left[\left(\int_0^T|\bar{Z}^n_t-Z^{n,m}_t|^2dt\right)^{\alpha/2}\right]=0, \\ &\hspace{2cm}\text{and}\quad\lim_{m\rightarrow\infty}\widehat{\mathbb{E}}\left[\sup_{t\in[0,T]}|\bar{A}^n_t-(A^{n,m,+}_t-K^{n,m}_t)|^\alpha\right]=0.
	\end{split}\end{equation}
By Lemma \ref{est-Y}, Lemma \ref{Y^{n,m}-U}, and Lemma \ref{A^{n,m,+},Z^{n,m},K^{n,m}}, together with Equation \eqref{statementB}, we have the following result.
\begin{lemma}\label{barY,barZ,barA}
	Assuming that $\xi$, $f$, $g$, $L$, and $U$ satisfy Assumptions (H2)-(H3) and (A1)-(A2). There exists a constant $C$ independent of $n$, such that for any $2\leq \alpha<\beta$, 
	\begin{align*}
		&\widehat{\mathbb{E}}\sup_{t\in[0,T]}|\bar{Y}^n_t|^\alpha\leq C, \qquad\qquad \widehat{\mathbb{E}}\left(\int_0^T |\bar{Z}^n_t|^2dt\right)^{\alpha/2}\leq C, \\ &\widehat{\mathbb{E}}\sup_{t\in[0,T]}|\bar{A}^{n}_t|^\alpha\leq C, \quad\;\text{and}\;\quad \widehat{\mathbb{E}}\sup_{t\in[0,T]}|(\bar{Y}^n_t-U_t)^+|^\alpha\leq \frac{C}{n^\alpha}.
	\end{align*}
\end{lemma}	

Finally, we investigate the difference of two solutions to the doubly reflected $G$-BSDE \eqref{eqn:DRGBSDE}.
\begin{proposition}\label{a prior estimate}
	Let $(Y^i,Z^i,A^i)$ for $i=1,2$  be the solutions to the doubly reflected $G$-BSDE \eqref{eqn:DRGBSDE} with parameters $(\xi^i,f^i,g^i,L^i,U^i)$, which satisfy Assumptions (H2)-(H3) and (A1)-(A2). Let $\{A^{i,n,+}\}_{n\in\mathbb{N}}$, $\{A^{i,n,-}\}_{n\in\mathbb{N}}$ and $\{K^{i,n}\}_{n\in\mathbb{N}}$ be the approximate sequences for $(Y^i,A^i)$ with order $\alpha$ w.r.t. $L^i$ and $U^i$, for $2\leq \alpha<\beta$. 
Set $$\widehat{Y}_t=Y^1_t-Y^2_t,  \ \widehat{\xi}=\xi^1-\xi^2,\ \widehat{L}_t=L^1_t-L^2_t\text{ and } \widehat{U}_t=U^1_t-U^2_t.$$
Then there exists a constant $C:=C(\alpha,T, \kappa,G)>0$ such that
\begin{align*}
	|\widehat{Y}_t|^\alpha\leq& C\left(\sum_{i=1}^2\widehat{\mathbb{E}}_t\sup_{s\in[t,T]}|Y^i_s|^\alpha\right)^{\frac{\alpha-2}{2}}\times\liminf_{n\rightarrow\infty}\left(\sum_{i=1}^2\left(\widehat{\mathbb{E}}_t|A^{i,n,+}_T|^\alpha+\widehat{\mathbb{E}}_t|A^{i,n,-}_T|^\alpha\right)\right)^{\frac{1}{\alpha}}\\
	&\times \left(\widehat{\mathbb{E}}_t\sup_{s\in[t,T]}|\widehat{L}_s|^\alpha+\widehat{\mathbb{E}}_t\sup_{s\in[t,T]}|\widehat{U}_s|^\alpha\right)^{\frac{1}{\alpha}}
	+C\widehat{\mathbb{E}}_t\left[|\widehat{\xi}|^\alpha+\int_t^T|\widehat{f}_s|^\alpha ds+\int_t^T|\widehat{g}_s|^\alpha ds\right],
\end{align*}
where
$\widehat{f}_s=\Big|f^1(s,Y_s^2,Z_s^2)-f^2(s,Y_s^2,Z_s^2)\Big|$ and $\widehat{g}_s=\Big|g^1(s,Y_s^2,Z_s^2)-g^2(s,Y_s^2,Z_s^2)\Big|$.
\end{proposition}

\begin{proof}
Set $$\widehat{Z}_t=Z_t^1-Z_t^2, \qquad \widehat{A}_t=A_t^1-A_t^2,\quad \text{and}\quad H_t=|\widehat{Y}_t|^2.$$ 
For any $r>0$, applying $G$-It\^{o}'s formula to $H_t^{\alpha/2}e^{rt}=|\widehat{Y}_t|^{\alpha} e^{rt}$, we have
\begin{align}\label{dr}
	&\hspace{-0.5cm}\quad H_t^{\alpha/2}e^{rt}+\int_t^T re^{rs}H_s^{\alpha/2}ds+\int_t^T \frac{\alpha}{2} e^{rs}
	H_s^{\alpha/2-1}\widehat{Z}_s^2d\langle B\rangle_s\nonumber\\
	&=|\widehat{\xi}|^\alpha e^{rT}+
	\alpha\Big(1-\frac{\alpha}{2}\Big)\int_t^Te^{rs}H_s^{\alpha/2-2}\widehat{Y}_s^2\widehat{Z}_s^2d\langle B\rangle_s-\int_t^T\alpha e^{rs}H_s^{\alpha/2-1}\widehat{Y}_s\widehat{Z}_sdB_s\nonumber\\
	&\quad+\int_t^T{\alpha} e^{rs}H_s^{\alpha/2-1}\widehat{Y}_s\Big(f^1(s,Y_s^1,Z_s^1)-f^2(s,Y_s^2,Z_s^2)\Big)ds +\int_t^T\alpha e^{rs}H_s^{\alpha/2-1}\widehat{Y}_sd\widehat{A}_s\\
	&\quad+\int_t^T{\alpha} e^{rs}H_s^{\alpha/2-1}\widehat{Y}_s\Big(g^1(s,Y_s^1,Z_s^1)-g^2(s,Y_s^2,Z_s^2)\Big)d\langle B\rangle_s.\nonumber
\end{align}
By the Lipschitz assumption on $f^1$ and $g^1$, together with H\"{o}lder's inequality and the fact that $\underline{\sigma}^2 ds\leq d\langle B\rangle_s\leq \bar{\sigma}^2ds$, we have
\begin{align*}
	&\hspace{-0.5cm}\int_t^T{\alpha} e^{rs}H_s^{\alpha/2-1}\widehat{Y}_s\Big(f^1(s,Y_s^1,Z_s^1)-f^2(s,Y_s^2,Z_s^2)\Big)ds\\
	&\hspace{-0.5cm}+\int_t^T{\alpha} e^{rs}H_s^{\alpha/2-1}\widehat{Y}_s\Big(g^1(s,Y_s^1,Z_s^1)-g^2(s,Y_s^2,Z_s^2)\Big)d\langle B\rangle_s\\
	\leq &\int_t^T{\alpha}e^{rs}H_s^{\frac{\alpha-1}{2}}\Big\{\Big|f^1(s,Y_s^1,Z_s^1)-f^1(s,Y_s^2,Z_s^2)\Big|+\widehat{f}_s\Big\}ds\\
	&+\int_t^T{\alpha}e^{rs}H_s^{\frac{\alpha-1}{2}}\Big\{\Big|g^1(s,Y_s^1,Z_s^1)-g^1(s,Y_s^2,Z_s^2)\Big|+\widehat{g}_s\Big\}d\langle B\rangle_s\\
	\leq &\int_t^T{\alpha}e^{rs}H_s^{\frac{\alpha-1}{2}}\Big\{\kappa\Big(|\widehat{Y}_s|+|\widehat{Z}_s|\Big)+\widehat{f}_s\Big\}ds+\int_t^T{\alpha}e^{rs}H_s^{\frac{\alpha-1}{2}}\Big\{\kappa\Big(|\widehat{Y}_s|+|\widehat{Z}_s|\Big)+\widehat{g}_s\Big\}d\langle B\rangle_s\\
	\leq &\int_t^T{\alpha} e^{rs}H_s^{\alpha/2-1/2}\Big(|\widehat{f}_s|+\bar{\sigma}^2|\widehat{g}_s|\Big)ds+\frac{\alpha(\alpha-1)}{4}\int_t^Te^{rs}H_s^{\alpha/2-1}\widehat{Z}_s^2d\langle B\rangle_s\\
	&+\left((1+\bar{\sigma}^2)\alpha \kappa+(1+\bar{\sigma}^4)\frac{2\alpha \kappa^2}{\underline{\sigma}^2(\alpha-1)}\right)\int_t^T e^{rs}H_s^{\alpha/2}ds.
\end{align*}
By Young's inequality, we obtain
\begin{align*}
	\int_t^T{\alpha} e^{rs}H_s^{\alpha/2-1/2}\Big(|\widehat{f}_s|+\bar{\sigma}^2|\widehat{g}_s|\Big)ds
	\leq &2(\alpha-1)\int_t^T  e^{rs}H_s^{\alpha/2}ds\\
 &+\int_t^T e^{rs}|\widehat{f}_s|^\alpha ds+\bar{\sigma}^{2\alpha}\int_t^T e^{rs}|\widehat{g}_s|^\alpha ds.
\end{align*}
Set $$A^{i,n}=A^{i,n,+}-A^{i,n,-}-K^{i,n}, \quad i=1,2,$$ $$\widehat{Y}^L_t=(Y^1_t-L^1_t)-(Y^2_t-L_t^2)\quad\text{and}\quad\widehat{Y}^U_t=(U^1_t-Y_t^1)-(U^2_t-Y_t^2).$$ Noting that 
\begin{align*}
\widehat{Y}_t=\widehat{Y}^L_t+\widehat{L}_t\leq Y^1_t-L^1_t+|\widehat{L}_t|,\qquad
-\widehat{Y}_t=\widehat{Y}^U_t-\widehat{U}_t \leq U_t^1-Y_t^1+|\widehat{U}_t|,
\end{align*}
and $A^{1,n,+}$, $A^{1,n,-}$ are non-decreasing processes, it is easy to check that
\begin{align*}
	\int_t^T\alpha e^{rs}H_s^{\alpha/2-1}\widehat{Y}_sdA^{1}_s
	=&\int_t^T\alpha e^{rs}H_s^{\alpha/2-1}\widehat{Y}_sd({A}^1_s-A_s^{1,n})+\int_t^T\alpha e^{rs}H_s^{\alpha/2-1}\widehat{Y}_sd{A}^{1,n}_s\\
	\leq &\int_t^T\alpha e^{rs}H_s^{\alpha/2-1}(Y^1_s-L^1_s)dA_s^{1,n,+}+\int_t^T \alpha e^{rs}H^{\alpha/2-1}_s|\widehat{L}_s|dA_s^{1,n,+}\\
	&+\int_t^T\alpha e^{rs}H_s^{\alpha/2-1}(U_s^1-Y_s^1)dA_s^{1,n,-}+\int_t^T \alpha e^{rs}H^{\alpha/2-1}_s|\widehat{U}_s|dA_s^{1,n,-}\\
	&+\left|\int_t^T\alpha e^{rs}H_s^{\alpha/2-1}\widehat{Y}_sd({A}^1_s-A_s^{1,n})\right|-\int_t^T\alpha e^{rs}H_s^{\alpha/2-1}(\widehat{Y}_s)^+d{K}^{1,n}_s.
\end{align*}
Similarly, we have
\begin{align*}
	\int_t^T\alpha e^{rs}H_s^{\alpha/2-1}(-\widehat{Y}_s)dA^{2}_s
	\leq &\int_t^T\alpha e^{rs}H_s^{\alpha/2-1}(Y^2_s-L^2_s)dA_s^{2,n,+}+\int_t^T \alpha e^{rs}H^{\alpha/2-1}_s|\widehat{L}_s|dA_s^{2,n,+}\\
	&+\int_t^T\alpha e^{rs}H_s^{\alpha/2-1}(U_s^2-Y_s^2)dA_s^{2,n,-}+\int_t^T \alpha e^{rs}H^{\alpha/2-1}_s|\widehat{U}_s|dA_s^{2,n,-}\\
	&+\left|\int_t^T\alpha e^{rs}H_s^{\alpha/2-1}\widehat{Y}_sd({A}^2_s-A_s^{2,n})\right|-\int_t^T\alpha e^{rs}H_s^{\alpha/2-1}(\widehat{Y}_s)^-d{K}^{2,n}_s.
\end{align*}
Since $|H_s^{\alpha/2-1}\widehat{Y}_s|\leq |\widehat{Y}_s|^{\alpha-1}$ for $s\in[0,T]$, it is easy to check that $H^{\alpha/2-1}\widehat{Y}\in S_G^{\frac{\alpha}{\alpha-1}}$. This fact, Lemma 3.1 in \cite{li2021backward} and
\begin{align*}
	\lim_{n\rightarrow\infty}\widehat{\mathbb{E}}\sup_{t\in[0,T]}|A^i_t-A^{i,n}_t|^\alpha=0,
\end{align*} imply that
\begin{displaymath}
	\lim_{n\rightarrow\infty}\widehat{\mathbb{E}}\left|\int_t^T\alpha e^{rs}H_s^{\alpha/2-1}\widehat{Y}_sd({A}^1_s-A_s^{1,n})\right|=0.
\end{displaymath}
Note that $U^1_t\geq Y^1_t$ and $A^{1,n,-}$ is non-decreasing. By the definition of $H$ and H\"{o}lder's inequality, it is easy to check that
\begin{align*}
	&\hspace{-1cm}\widehat{\mathbb{E}}\left[\int_t^T\alpha e^{rs}H_s^{\alpha/2-1}(U_s^1-Y_s^1)dA_s^{1,n,-}\right]\\
	\leq &C\widehat{\mathbb{E}}\left[\sup_{t\in[0,T]}(|Y_t^1|+|Y_t^2|)^{\alpha-2}\int_t^T (U_s^1-Y_s^1)d A_s^{1,n,-}\right]\\
	\leq& C\left(\widehat{\mathbb{E}}\sup_{t\in[0,T]}\Big(|Y_t^1|^\alpha+|Y_t^2|^\alpha\Big)\right)^{\frac{\alpha-2}{\alpha}}\left(\widehat{\mathbb{E}}\left|\int_t^T (U_s^1-Y_s^1)d A_s^{1,n,-}\right|^{\frac{\alpha}{2}}\right)^{\frac{2}{\alpha}}.
\end{align*}
It follows from the $\textmd{ASC}_{\alpha}$ that
\begin{displaymath}
	\lim_{n\rightarrow\infty}\widehat{\mathbb{E}}\left|\int_t^T\alpha e^{rs}H_s^{\alpha/2-1}(U_s^1-Y_s^1)dA_s^{1,n,-}\right|=0.
\end{displaymath}
Similar analyses yield that
\begin{align*}
	&\lim_{n\rightarrow\infty}\widehat{\mathbb{E}}\left|\int_t^T\alpha e^{rs}H_s^{\alpha/2-1}(Y_s^1-L_s^1)dA_s^{1,n,+}\right|=0,\\
	&\lim_{n\rightarrow\infty}\widehat{\mathbb{E}}\left|\int_t^T\alpha e^{rs}H_s^{\alpha/2-1}(Y_s^2-L_s^2)dA_s^{2,n,+}\right|=0,\\
	&\lim_{n\rightarrow\infty}\widehat{\mathbb{E}}\left|\int_t^T\alpha e^{rs}H_s^{\alpha/2-1}(U_s^2-Y_s^2)dA_s^{2,n,-}\right|=0.
\end{align*}
By the non-decreasing property of $A^{1,n,+}$, the definition of $H$, and H\"{o}lder's inequality, we obtain that
\begin{align*}
	\widehat{\mathbb{E}}_t\left[\int_t^T \alpha e^{rs}H^{\alpha/2-1}_s|\widehat{L}_s|dA_s^{1,n,+}\right]
	\leq& C\widehat{\mathbb{E}}_t\left[\sup_{s\in[t,T]}(|Y_s^1|+|Y_s^2|)^{\alpha-2}\sup_{s\in[t,T]}|\widehat{L}_s||A^{1,n,+}_T|\right]\\
	\leq &CI_t^\alpha(Y)\bigg(\widehat{\mathbb{E}}_t\sup_{s\in[t,T]}|\widehat{L}_s|^\alpha\bigg)^{\frac{1}{\alpha}}\bigg(\widehat{\mathbb{E}}_t|A_T^{1,n,+}|^\alpha]\bigg)^{\frac{1}{\alpha}},
\end{align*}
where $$I_t^\alpha(Y)=\left(\sum_{i=1}^2\widehat{\mathbb{E}}_t\sup_{s\in[t,T]}|Y_s^i|^\alpha\right)^{\frac{\alpha-2}{\alpha}}.$$ Similarly, we have 	\begin{align*}
	&\widehat{\mathbb{E}}_t\left[\int_t^T \alpha e^{rs}H^{\alpha/2-1}_s|\widehat{U}_s|dA_s^{1,n,-}\right]
	\leq CI_t^\alpha(Y)\bigg(\widehat{\mathbb{E}}_t\sup_{s\in[t,T]}|\widehat{U}_s|^\alpha\bigg)^{\frac{1}{\alpha}}\bigg(\widehat{\mathbb{E}}_t|A_T^{1,n,-}|^\alpha\bigg)^{\frac{1}{\alpha}},\\
	&\widehat{\mathbb{E}}_t\left[\int_t^T \alpha e^{rs}H^{\alpha/2-1}_s|\widehat{L}_s|dA_s^{2,n,+}\right]
	\leq  CI_t^\alpha(Y)\bigg(\widehat{\mathbb{E}}_t\sup_{s\in[t,T]}|\widehat{L}_s|^\alpha\bigg)^{\frac{1}{\alpha}}\bigg(\widehat{\mathbb{E}}_t|A_T^{2,n,+}|^\alpha\bigg)^{\frac{1}{\alpha}},\\
	&\widehat{\mathbb{E}}_t\left[\int_t^T \alpha e^{rs}H^{\alpha/2-1}_s|\widehat{U}_s|dA_s^{2,n,-}\right]
	\leq  CI_t^\alpha(Y)\bigg(\widehat{\mathbb{E}}_t\sup_{s\in[t,T]}|\widehat{U}_s|^\alpha\bigg)^{\frac{1}{\alpha}}\bigg(\widehat{\mathbb{E}}_t|A_T^{2,n,-}|^\alpha\bigg)^{\frac{1}{\alpha}}.
\end{align*}
Set $$M^n_t=\int_0^t\alpha e^{rs}H_s^{\alpha/2-1}\Big(\widehat{Y}_s\widehat{Z}_sdB_s+(\widehat{Y}_s)^{+}dK^{1,n}_s+(\widehat{Y}_s)^-dK_s^{2,n}\Big).$$ By Lemma 3.4 in \cite{hu2014backward},  $M^n$ is a $G$-martingale. Let 
$$
r=2(\alpha-1)+(1+\bar{\sigma}^2)\alpha \kappa+(1+\bar{\sigma}^4)\frac{2\alpha \kappa^2}{\underline{\sigma}^2(\alpha-1)}+1.
$$
Combining the above inequalities, we obtain
\begin{align*}
	&\hspace{-0.5cm}H_t^{\alpha/2}e^{rt}+(M^n_T-M^n_t)\\ \leq &|\widehat{\xi}|^{\alpha} e^{rT}+\int_t^T e^{rs}|\widehat{f}_s|^\alpha ds+\bar{\sigma}^{2\alpha}\int_t^T e^{rs}|\widehat{g}_s|^\alpha ds+\sum_{i=1}^2\left|\int_t^T\alpha e^{rs}H_s^{\alpha/2-1}\widehat{Y}_sd({A}^i_s-A_s^{i,n})\right|\\
	&+\int_t^T\alpha e^{rs}H_s^{\alpha/2-1}|\widehat{L}_s|d(A_s^{1,n,+}+A_s^{2,n,+})
	+\int_t^T\alpha e^{rs}H_s^{\alpha/2-1}|\widehat{U}_s|d(A_s^{1,n,-}+A_s^{2,n,-})\\
	&+\sum_{i=1}^2 \int_t^T\alpha e^{rs}H_s^{\alpha/2-1}(U_s^i-Y_s^i)dA_s^{i,n,-}+\sum_{i=1}^2 \int_t^T\alpha e^{rs}H_s^{\alpha/2-1}(Y_s^i-L_s^i)dA_s^{i,n,+}.
\end{align*}
Taking conditional expectations on both sides and letting $n\rightarrow \infty$, the proof is complete.
\end{proof}

\subsection{Proof of Theorem \ref{main}}
\label{sec:mainproof}
\noindent	\textbf{(a).}	We first prove the uniqueness of the solution to the doubly reflected $G$-BSDE \eqref{eqn:DRGBSDE}. Let $(Y^i,Z^i,A^i)$ for $i=1,2$ be the solutions to the doubly reflected $G$-BSDE \eqref{eqn:DRGBSDE}. By Proposition \ref{a prior estimate}, we conclude that $Y^1\equiv Y^2$. Applying $G$-It\^{o}'s formula to $(Y^1_t-Y^2_t)^2$, we obtain that 
\begin{align*}
	\int_t^T |Z^1_s-Z^2_s|^2d\langle B\rangle_s=&-(Y^1_t-Y^2_t)^2+\int_t^T (Y^1_s-Y^2_s)\Big(f(s,Y^1_s,Z^1_s)-f(s,Y^2_s,Z^2_s)\Big)ds\\
	&+\int_t^T (Y^1_s-Y^2_s)\Big(g(s,Y^1_s,Z^1_s)-g(s,Y^2_s,Z^2_s)\Big)d\langle B\rangle_s\\
	&-\int_t^T 2(Y^1_s-Y^2_s)(Z^1_s-Z^2_s)dB_s+\int_t^T 2(Y^1_s-Y^2_s)d(A^1_s-A^2_s).
\end{align*}
Using the fact that $Y^1\equiv Y^2$, taking $t=0$ in the above equation, it is easy to check that 
\begin{align*}
	\widehat{\mathbb{E}}\left(\int_0^T |Z^1_s-Z^2_s|^2d\langle B\rangle_s\right)^{\alpha/2}=0.
\end{align*}
Since $\underline{\sigma}^2>0$, it follows that $Z^1\equiv Z^2$. Note that for $i=1,2$,
\begin{align*}
	A^i_t=Y^i_0-Y^i_t-\int_0^t f(s,Y^i_s,Z^i_s)ds-\int_0^t g(s,Y^i_s,Z^i_s)d\langle B\rangle_s+\int_0^t Z^i_s dB_s.
\end{align*}
Applying the Lipschitz assumption on $f,g$, H\"{o}lder's inequality and Proposition \ref{the1.3}, we have 
\begin{align*}
	\widehat{\mathbb{E}}\sup_{t\in[0,T]}|A^1_t-A^2_t|^\alpha\leq C\widehat{\mathbb{E}}\sup_{t\in[0,T]}|Y^1_t-Y^2_t|^\alpha+C\widehat{\mathbb{E}}\left(\int_0^T |Z^1_s-Z^2_s|^2ds\right)^{\alpha/2}=0,
\end{align*}
which implies that $A^1\equiv A^2$.

Then, we prove the existence of the solution to the doubly reflected $G$-BSDE \eqref{eqn:DRGBSDE}. Letting $m=n$ in Equation \eqref{Y^{n,m}}, we define
	\begin{align}
		\label{eqn:Ann}
	Y^n=Y^{n,n}, \ Z^n=Z^{n,n}, \ K^n=K^{n,n}, \	A^{n,-}=A^{n,n,-}\text{ and } A^{n,+}=A^{n,n,+}.
	\end{align}
 Set $$A^n=A^{n,-}-K^n-A^{n,+}.$$ 
 By a similar analysis as the proof of Lemma 4.4 and Lemma 4.7 in \cite{li2021backward}, we have for any $2\leq \alpha<\beta$, 
    \begin{align}\label{statementC}
        \lim_{n\rightarrow \infty}\hE\left[\sup_{t\in[0,T]}|(Y^n_t-L_t)^-|^\alpha\right]=0
    \end{align}
and
    \begin{align*}
		&\lim_{n,n'\rightarrow\infty}\widehat{\mathbb{E}}\left[\sup_{t\in[0,T]}|Y_t^n-Y_t^{n'}|^\alpha\right]=0,\qquad\lim_{n,n'\rightarrow\infty}\widehat{\mathbb{E}}\left[\left(\int_0^T|Z_s^n-Z_s^{n'}|^2ds\right)^{\frac{\alpha}{2}}\right]=0, \\
		&\hspace{2cm}\text{and}\quad\lim_{n,n'\rightarrow\infty}\widehat{\mathbb{E}}\left[\sup_{t\in[0,T]}|A_t^n-A_t^{n'}|^\alpha\right]=0.
	\end{align*}
	Denote by $(Y,Z,A)$ the limit of $(Y^n,Z^n,A^n)$ as $n$ goes to infinity. Recalling the definitions of $A^{n,+}$ and $A^{n,-}$ given in Equation \eqref{eqn:Ann}, and the fact that  $\widehat{\mathbb{E}}|A^{n,+}|^\alpha]\leq C$ and $\widehat{\mathbb{E}}|A^{n,-}|^\alpha]\leq C$ from Lemma \ref{A^{n,m,+},Z^{n,m},K^{n,m}}, 
 we have $L_t\leq Y_t\leq U_t$ for $t\in[0,T]$. Letting $n\rightarrow \infty$ in Equation \eqref{Y^{n,m}} (recalling here we consider the case that $m=n$) yields
 \begin{align*}
Y_t=\xi+\int_t^T f(s,Y_s,Z_s)ds+\int_t^T g(s,Y_s,Z_s)d\langle B\rangle_s-\int_t^T Z_s dB_s+(A_T-A_t).
 \end{align*} 
 It remains to prove that $(Y,A)$ satisfies the $\textmd{ASC}_{\alpha}$. We claim that $\{A^{n,+}\}_{n\in\mathbb{N}}$, $\{A^{n,-}\}_{n\in\mathbb{N}}$ and $\{K^n\}_{n\in\mathbb{N}}$ are the approximate sequences for $(Y,A)$ with order $\alpha$. It suffices to show that 
 \begin{align*}
     \lim_{n\rightarrow\infty}\widehat{\mathbb{E}}\left|\int_0^T (Y_s-L_s)d A_s^{n,+}\right|^{\alpha/2}=0\quad\text{and}\quad\lim_{n\rightarrow\infty}\widehat{\mathbb{E}}\left|\int_0^T (U_s-Y_s)d A_s^{n,-}\right|^{\alpha/2}=0.
 \end{align*}
 We only prove the first equation since the second one can be proved similarly. By the definition of $A^{n,+}$ given in Equation \eqref{eqn:Ann}, we obtain that 
 \begin{align*}
   \int_0^T (Y_s-L_s)d A_s^{n,+}&=\int_0^T (Y_s-Y^n_s)d A_s^{n,+}+\int_0^T (Y^n_s-L_s)n(Y^n_s-L_s)^-ds\\
   &\leq \sup_{t\in[0,T]}|Y_t-Y^n_t||A_T^{n,+}|.
 \end{align*}
 Then, it is easy to check that 
 \begin{align*}
     \lim_{n\rightarrow\infty}\widehat{\mathbb{E}}\left|\int_0^T (Y_s-L_s)d A_s^{n,+}\right|^{\alpha/2}\leq \lim_{n\rightarrow \infty}\left(\widehat{\mathbb{E}}\sup_{t\in[0,T]}|Y_t-Y^n_t|^\alpha\right)^{\frac{1}{2}}\left(\widehat{\mathbb{E}}|A^{n,+}_T|\right)^{\frac{1}{2}}=0.
 \end{align*}
 Therefore, $(Y,Z,A)$ is the solution to the doubly reflected $G$-BSDE \eqref{eqn:DRGBSDE}.\\
 
 \noindent	\textbf{(b).} Next, we demonstrate the decreasing convergence of $\bar{Y}^n$ to $Y$.	By Theorem \ref{the1.16}, we have $\bar{Y}^{n_1}_t\geq \bar{Y}^{n_2}_t$  for any $n_1\leq n_2$ and $t\in[0,T]$. It suffices to show that for any $2\leq \alpha<\beta$,
	\begin{equation}\label{e3}
		\lim_{n\rightarrow\infty}\widehat{\mathbb{E}}\sup_{t\in[0,T]}|Y^n_t-\bar{Y}^n_t|^\alpha=0.
	\end{equation}
	Noting that $\bar{Y}^n_t\geq L_t$ for any $n\in\mathbb{N}$ and any $t\in[0,T]$, $(\bar{Y}^n,\bar{Z}^n,\bar{A}^n)$ satisfies the following equation
	\begin{align*}
		\bar{Y}_t^n=&\xi+\int_t^T f(s,\bar{Y}^n_s,\bar{Z}_s^n)ds+\int_t^T g(s,\bar{Y}^n_s,\bar{Z}_s^n)d\langle B\rangle_s-\int_t^T \bar{Z}_s^ndB_s\\
		&+(\bar{A}^n_T-\bar{A}^n_t)-\int_t^T n(\bar{Y}^n_s-U_s)^+ds+\int_t^Tn(\bar{Y}^n_s-L_s)^-ds.
	\end{align*}
	Additionally, since  $L_t\leq \bar{Y}^n_t\leq \bar{Y}^1_t$ for any $n\in\mathbb{N}$ and $t\in[0,T]$,  there exists a constant $C$ independent of $n$, such that for any $2\leq \alpha<\beta$,
    \begin{align}\label{est-barYn}
        \hE\sup_{t\in[0,T]}|\bar{Y}^n_t|^\alpha\leq C.
    \end{align} 
    By Theorem \ref{the1.5},  $$\widehat{Y}^n_t=\bar{Y}^n_t-Y_t^n\geq 0$$
     for any $n\in\mathbb{N}$ and $t\in[0,T]$. For any constant $r$, applying $G$-It\^{o}'s formula to $e^{rt}(H^n_t)^{\frac{\alpha}{2}}$, where $H^n_t=|\widehat{Y}^n_t|^2$, we have
	\begin{equation}\label{e1}
		\begin{split}
			&\hspace{-0.5cm} |H_t^n|^{\alpha/2}e^{rt}+\int_t^T re^{rs}|H^n_s|^{\alpha/2}ds+\int_t^T \frac{\alpha}{2} e^{rs}
			|H_s^n|^{\alpha/2-1}(\widehat{Z}^n_s)^2d\langle B\rangle_s\\
			=&
			\alpha(1-\frac{\alpha}{2})\int_t^Te^{rs}|H_s^n|^{\alpha/2-2}(\widehat{Y}^n_s)^2(\widehat{Z}^n_s)^2d\langle B\rangle_s
			-\int_t^T\alpha e^{rs}|H_s^n|^{\alpha/2-1}\widehat{Y}_sn(Y_s^n-L_s)^-ds\\
			&+\int_t^T{\alpha} e^{rs}|H_s^n|^{\alpha/2-1}\widehat{Y}_s\widehat{f}^n_sds+\int_t^T{\alpha} e^{rs}|H_s^n|^{\alpha/2-1}\widehat{Y}_s\widehat{g}^n_sd\langle B\rangle_s\\
            &-\int_t^T\alpha e^{rs}|H_s^n|^{\alpha/2-1}\Big(\widehat{Y}^n_s\widehat{Z}^n_sdB_s-\widehat{Y}^n_sd{K}^n_s-\widehat{Y}^n_sd\bar{A}^n_s\Big)\\
			&-\int_t^T \alpha e^{rs}|H_s^n|^{\alpha/2-1}\widehat{Y}_sn\Big[(\bar{Y}^n_s-U_s)^+-(Y^n_s-U_s)^+\Big]ds,
		\end{split}
	\end{equation}
	where $$\widehat{Z}^n_t=\bar{Z}^n_t-Z^n_t,\quad \widehat{f}^n_t=f(t,\bar{Y}^n_t,\bar{Z}^n_t)-f(t,Y^n_t,Z^n_t) \quad \text{and}\quad \widehat{g}^n_t=g(t,\bar{Y}^n_t,\bar{Z}^n_t)-g(t,Y^n_t,Z^n_t).$$ Applying H\"{o}lder's inequality, we have
	\begin{align*}
		&\int_t^T{\alpha} e^{rs}|H_s^n|^{\alpha/2-1}\widehat{Y}^n_s\widehat{f}^n_sds+\int_t^T{\alpha} e^{rs}|H_s^n|^{\alpha/2-1}\widehat{Y}^n_s\widehat{g}^n_sd\langle B\rangle_s\\
		&\leq \left((1+\bar{\sigma}^2)\alpha \kappa+(1+\bar{\sigma}^4)\frac{2\alpha \kappa^2}{\underline{\sigma}^2(\alpha-1)}\right)\int_t^T e^{rs}|H_s^n|^{\alpha/2}ds\\
		&\hspace{5cm}
		 +\frac{\alpha(\alpha-1)}{4}\int_t^Te^{rs}|H_s^n|^{\alpha/2-1}(\widehat{Z}^n_s)^2d\langle B\rangle_s.
	\end{align*}
	Noting that $\widehat{Y}^n_t\geq0$ and $\bar{A}^n$ is non-decreasing, it is easy to check that
	\begin{equation}\label{e2}\begin{split}
		\bullet	\quad&+\int_t^T\alpha e^{rs}|H_s^n|^{\alpha/2-1}\widehat{Y}^n_sd\bar{A}^n_s\leq \int_t^T\alpha e^{rs}|H_s^n|^{\alpha/2-1}\Big[(\bar{Y}^n_s-L_s)+(Y^n_s-L_s)^-\Big]d\bar{A}^n_s,\\
		\bullet	\quad&-\int_t^T \alpha e^{rs}|H_s^n|^{\alpha/2-1}\widehat{Y}_sn\Big[(\bar{Y}^n_s-U_s)^+-(Y^n-U_s)^+\Big]ds\leq 0,\\
		\bullet	\quad&-\int_t^T\alpha e^{rs}|H_s^n|^{\alpha/2-1}\widehat{Y}_sn(Y_s^n-L_s)^-ds\leq 0,\\
		\bullet	\quad&+\int_t^T\alpha e^{rs}|H_s^n|^{\alpha/2-1}\widehat{Y}^n_sd{K}^n_s\leq 0.
	\end{split} \end{equation}
	Set 
	$$M^n_t=\int_0^t\alpha e^{rs}|H_s^n|^{\alpha/2-1}\Big(\widehat{Y}^n_s\widehat{Z}^n_sdB_s-(\bar{Y}^n_s-L_s)d\bar{A}^n_s\Big),$$
	 which is a $G$-martingale. Letting $$r=1+\left((1+\bar{\sigma}^2)\alpha \kappa+(1+\bar{\sigma}^4)\frac{2\alpha \kappa^2}{\underline{\sigma}^2(\alpha-1)}\right),$$ 
	 all the above analyses indicate that
	\begin{displaymath}
		e^{rt}|\widehat{Y}^n_t|^\alpha+(M^n_T-M^n_t)\leq \int_t^T\alpha e^{rs}|H_s^n|^{\alpha/2-1}(Y^n_s-L_s)^-d\bar{A}^n_s.
	\end{displaymath}
	Taking conditional expectations on both sides, we have
	\begin{displaymath}
		|\widehat{Y}^n_t|^\alpha\leq C\widehat{\mathbb{E}}_t\left[\int_t^T|H_s^n|^{\alpha/2-1}(Y^n_s-L_s)^-d\bar{A}^n_s\right].
	\end{displaymath}
	Thanks to Theorem \ref{the1.2}, to obtain Equation \eqref{e3},  it suffices to show that there exists some $\gamma>1$, such that
	\begin{displaymath}
		\lim_{n\rightarrow\infty}\widehat{\mathbb{E}}\left(\int_0^T|H_s^n|^{\alpha/2-1}(Y^n_s-L_s)^-d\bar{A}^n_s\right)^\gamma=0.
	\end{displaymath}
	Indeed, for any $1<\gamma<\beta/\alpha$, we have
	\begin{align*}
		&\hspace{-1cm}\widehat{\mathbb{E}}\left(\int_0^T |H_s^n|^{\alpha/2-1}(Y_s^n-L_s)^-d\bar{A}^n_s\right)^{\gamma}\\
		\le&\widehat{\mathbb{E}}\left[\sup_{s\in[0,T]}|\widehat{Y}^n_s|^{(\alpha-2)\gamma}\sup_{s\in[0,T]}\big((Y_s^n-L_s)^-\big)^{\gamma}\big(\bar{A}^{n}_T\big)^{\gamma}\right]\\
		\le&\bigg(\widehat{\mathbb{E}}\sup_{s\in[0,T]}|\widehat{Y}^n_s|^{\alpha\gamma}\bigg)^{\frac{\alpha}{\alpha-2}}\bigg(\widehat{\mathbb{E}}\sup_{s\in[0,T]}\big((Y_s^n-L_s)^-\big)^{\alpha\gamma}	\bigg)^{\frac{1}{\alpha}}\widehat{\mathbb{E}}\Big[\big(\bar{A}^{n}_T\big)^{\alpha\gamma}\Big]^{\frac{1}{\alpha}},
	\end{align*}
	which converges to zero as $n$ goes to infinity, by Lemmas \ref{est-Y} and \ref{barY,barZ,barA}, and Equations \eqref{statementC} and \eqref{est-barYn}.\\
	
	\noindent	\textbf{(c).} 
	In order to prove the last assertion in Theorem \ref{main}, it suffices to show that for any $2\leq \alpha<\beta$, we have 
	\begin{displaymath}
		\lim_{n\rightarrow\infty}\widehat{\mathbb{E}}\left(\int_0^T |\bar{Z}^n_s-Z_s^n|^2ds\right)^{\frac{\alpha}{2}}=0\quad\text{and}\quad \lim_{n\rightarrow\infty}\widehat{\mathbb{E}}\sup_{t\in[0,T]}|\widetilde{A}^n_t-A^n_t|^\alpha=0,
	\end{displaymath}
	where $\widetilde{A}^n_t=\bar{A}^n_t-\int_0^t n(\bar{Y}^n_s-U_s)^+ds$. 
	
	Letting $r=0$ and $\alpha=2$ in Equation \eqref{e1}, applying Equation \eqref{e2}, we have
	\begin{align*}
		\int_0^T (\widehat{Z}^n_s)^2d\langle B\rangle_s\leq &\int_0^T 2\widehat{Y}^n_s\widehat{f}^n_sds+\int_0^T 2\widehat{Y}^n_s\widehat{g}^n_sd\langle B\rangle_s-\int_0^T 2\widehat{Y}^n_s\widehat{Z}_s^ndB_s+\int_0^T 2\widehat{Y}_s^nd\bar{A}^n_s\\
		\leq &\kappa_\varepsilon\int_0^T (\widehat{Y}^n_s)^2ds+2\varepsilon\int_0^T(\widehat{Z}^n_s)^2ds+2\sup_{t\in[0,T]}|\widehat{Y}^n_s||\bar{A}^n_T|-\int_0^T 2\widehat{Y}^n_s\widehat{Z}_s^ndB_s,
	\end{align*}
	where $\varepsilon>0$ and $\kappa_\varepsilon=2(1+\bar{\sigma}^2)\kappa+\frac{(1+\bar{\sigma}^4)\kappa^2}{\varepsilon}$. By Proposition \ref{the1.3}, for any $\varepsilon'>0$, we obtain
	\begin{align*}
		\widehat{\mathbb{E}}\left(\int_0^T \widehat{Y}^n_s\widehat{Z}^n_sdB_s\right)^{\frac{\alpha}{2}}\leq &C\widehat{\mathbb{E}}\left(\int_0^T (\widehat{Y}_s^n)^2(\widehat{Z}_s^n)^2 ds\right)^{\frac{\alpha}{4}}\\
		\leq&C\left(\widehat{\mathbb{E}}\sup_{t\in[0,T]}|\widehat{Y}^n_t|^\alpha\right)^{1/2}\left(\widehat{\mathbb{E}}\left(\int_0^T |\widehat{Z}_s^n|^2ds\right)^{\frac{\alpha}{2}}\right)^{1/2}\\
		\leq &\frac{C}{4\varepsilon'}\widehat{\mathbb{E}}\sup_{t\in[0,T]}|\widehat{Y}^n_t|^\alpha+C\varepsilon'\widehat{\mathbb{E}}\left(\int_0^T |\widehat{Z}_s^n|^2ds\right)^{\frac{\alpha}{2}}.
	\end{align*}
	Choosing $\varepsilon$ and $\varepsilon'$ small enough, it is easy to check that
	\begin{displaymath}
		\widehat{\mathbb{E}}\left(\int_0^T (\widehat{Z}^n_s)^2ds\right)^{\frac{\alpha}{2}}\leq C\left\{\widehat{\mathbb{E}}\sup_{t\in[0,T]}|\widehat{Y}^n_t|^\alpha+\bigg(\widehat{\mathbb{E}}\sup_{t\in[0,T]}|\widehat{Y}^n_t|^\alpha\bigg)^{1/2}\left(\widehat{\mathbb{E}}|\bar{A}^n_T|^\alpha\right)^{1/2}\right\}.
	\end{displaymath}
	It follows from Lemma \ref{barY,barZ,barA} and Equation \eqref{e3} that $$\lim_{n\rightarrow\infty}\widehat{\mathbb{E}}\left(\int_0^T |\bar{Z}^n_s-Z_s^n|^2ds\right)^{\frac{\alpha}{2}}=0.$$ 
	Finally, we have
	\begin{displaymath}
		\begin{split}
			&\widehat{\mathbb{E}}\sup_{t\in[0,T]}|\widetilde{A}_t^n-A_t^n|^\alpha\\
			&\leq C\widehat{\mathbb{E}}\left[\sup_{t\in[0,T]}|\widehat{Y}^n_t|^\alpha+\left(\int_0^T|\widehat{f}^n_s| ds\right)^\alpha+\left(\int_0^T|\widehat{g}^n_s| ds\right)^\alpha
			+\sup_{t\in[0,T]}\left|\int_0^t \widehat{Z}^n_sdB_s\right|^\alpha\right]\\
			&\leq C\left\{\widehat{\mathbb{E}}\sup_{t\in[0,T]}|\widehat{Y}^n_t|^\alpha+\widehat{\mathbb{E}}\left(\int_0^T|\widehat{Z}^n_s|^2ds\right)^{\alpha/2}\right\}\\
			&\rightarrow 0,\qquad\textrm{ as }n\rightarrow \infty.
		\end{split}
	\end{displaymath}
	The proof is complete.

\section{Probabilistic representation of fully nonlinear PDEs with double obstacles}
\label{sec:PDE}
In this section, we establish the connection between fully nonlinear PDEs with double obstacles and doubly reflected $G$-BSDEs. To this end, we consider the doubly reflected $G$-BSDEs in a Markovian framework. For simplicity, we focus solely on doubly reflected BSDEs driven by one-dimensional $G$-Brownian motion. However, similar results apply to the multi-dimensional case.

For each $0\leq t\leq T$ and $\xi\in L_G^p(\Omega_t)$ where $p\geq 2$, let $\{X_s^{t,\xi}, t\leq s\leq T\}$ be the  solution of the following $G$-SDE:
\begin{equation}\label{eq1.6}
	X_s^{t,\xi}=x+\int_t^s b(r,X_r^{t,\xi})dr+\int_t^s l(r,X_r^{t,\xi})d\langle B\rangle_r+\int_t^s \sigma(r,X_r^{t,\xi})dB_r.
\end{equation}
Consider the doubly reflected $G$-BSDE
\begin{align}
	\label{eqn:DRGBSDE_PDE}
	\begin{cases}
		Y_s^{t,x}=\xi^{t,x}+\int_s^T f^{t,x}(s,Y_s^{t,x},Z_s^{t,x})ds+\int_s^T g^{t,x}(s,Y_s^{t,x},Z_s^{t,x})d\langle B\rangle_s\\
		\hspace{4cm}-\int_s^T Z_s^{t,x} dB_s+(A_s^{t,x}-A_s^{t,x}), \hspace{1cm} t\leq s\leq T,\vspace{0.2cm}\\
		L^{t,x}_s\leq Y_s^{t,x}\leq U^{t,x}_s,\quad t\leq s\leq T, \vspace{0.2cm}\\
		(Y^{t,x}, A^{t,x}) \textrm{ satisfies the } \textmd{ASC}_{\alpha},
	\end{cases}
\end{align}
which is the doubly reflected $G$-BSDE \eqref{eqn:DRGBSDE} 
with parameters $(\xi^{t,x},f^{t,x},g^{t,x},L^{t,x}, U^{t,x})$ taking the following form:
\begin{align*}
\xi^{t,x}=\phi(X_T^{t,x}),\qquad L_s^{t,x}=h(s,X_s^{t,x}), \qquad U_s^{t,x}=h'(s,X_s^{t,x}),\\
    f^{t,x}(s,y,z)=f(s,X_s^{t,x},y,z), \qquad\text{and}\qquad g^{t,x}(s,y,z)=g(s,X_s^{t,x},y,z).
\end{align*}
The functions  $b,l,\sigma,h,h':[0,T]\times\mathbb{R}\rightarrow \mathbb{R}$, $\phi:\mathbb{R}\rightarrow \mathbb{R}$ and  $f,g:[0,T]\times\mathbb{R}^3\rightarrow \mathbb{R}$  are assumed to be deterministic and satisfy the following conditions:
\begin{description}
	\item[(Ai)] $b$, $l$, $\sigma$, $f$, $g$, $h$, $h'$ are continuous in $t\in [0,T]$;
	\item[(Aii)] There exist a positive integer $k$ and a constant $\kappa$ such that
	\begin{align*}
		&|\phi(x)-\phi(x')|\leq \kappa(1+|x|^k+|x'|^k)|x-x'|,\\
		&|f(t,x,y,z)-f(t,x',y',z')|\leq \kappa\Big[(1+|x|^k+|x'|^k)|x-x'|+|y-y'|+|z-z'|\Big],\\
        &|g(t,x,y,z)-g(t,x',y',z')|\leq \kappa\Big[(1+|x|^k+|x'|^k)|x-x'|+|y-y'|+|z-z'|\Big],\\
        &|b(t,x)-b(t,x')|+|l(t,x)-l(t,x')|+|\sigma(t,x)-\sigma(t,x')|+|h(t,x)-h(t,x')|\leq \kappa|x-x'|;
	\end{align*}
	\item[(Aiii)]$h'$ belongs to the space $C^{1,2}_{Lip}([0,T]\times \mathbb{R})$, $h(t,x)\leq h'(t,x)$ and $h(T,x)\leq \phi(x)\leq h'(T,x)$ for any $x\in\mathbb{R}$ and $t\in[0,T]$. The space $C^{1,2}_{Lip}([0,T]\times \mathbb{R})$ refers to the space of functions that are continuously differentiable in their first variable and twice continuously differentiable in their second variable, and both derivatives are uniformly Lipschitz continuous. 
\end{description}

Under the above conditions, the solutions of the $G$-SDE \eqref{eq1.6} have the following properties; see Chapter V of \cite{peng2019nonlinear}.
\begin{proposition}[\cite{peng2019nonlinear}]\label{the1.17}
	Let $\xi,\xi'\in L_G^p(\Omega_t)$ where $p\geq 2$. Then we have, for each $\delta\in[0,T-t]$,
	\begin{align*}
		\widehat{\mathbb{E}}_t\sup_{s\in[t,t+\delta]}&|X_{s}^{t,\xi}-X_{s}^{t,\xi'}|^p\leq C|\xi-\xi'|^p,\qquad
		\widehat{\mathbb{E}}_t|X_{t+\delta}^{t,\xi}|^p\leq C(1+|\xi|^p),\\
	&\quad	\text{and}\quad\widehat{\mathbb{E}}_t\sup_{s\in[t,t+\delta]}|X_s^{t,\xi}-\xi|^p\leq C(1+|\xi|^p)\delta^{p/2},
	\end{align*}
	where the constant $C$ depends on $\kappa,G,p$ and $T$.
\end{proposition}

Now define
\begin{equation}\label{eq1.8}
	u(t,x):=Y_t^{t,x},\quad (t,x)\in[0,T]\times\mathbb{R},
\end{equation}
where $Y^{t,x}$ is the first component of the solution to the doubly reflected $G$-BSDE \eqref{eqn:DRGBSDE_PDE}. Our first observation is that $u$ is a deterministic and continuous function. 

\begin{lemma}\label{l1}
	For any fixed $t\in[0,T]$, $u$ is a continuous function in $x$.
\end{lemma}

\begin{proof}
	By Proposition \ref{a prior estimate} and Proposition \ref{the1.17}, there exists a constant $C$ depending on $T,k,\kappa,G,x,x'$, such that for any $t\in[0,T]$ and  $x,x'\in\mathbb{R}$,
	\begin{displaymath}
		|u(t,x)-u(t,x')|^2\leq C(|x-x'|^2+|x-x'|).
	\end{displaymath}
	This completes the proof.
\end{proof}

\begin{lemma}\label{l2}
	For any fixed $x\in\mathbb{R}$, $u$ is continuous in $t$.
\end{lemma}

\begin{proof}
	For any fixed $t\in[0,T]$, we define, for $0\leq s\leq t$,
	$$X_s^{t,x}:=x,\quad Y_s^{t,x}:=Y_t^{t,x},\quad Z_s^{t,x}:= 0,\quad A_s^{t,x}:=0,$$ 
	$$U^{t,x}_s:=h'(t,x) \quad \text{and} \quad L^{t,x}_s:=h(t,x).$$  Obviously, $(Y^{t,x}_s,Z^{t,x}_s,A^{t,x}_s)_{s\in[0,T]}$ is the solution to the doubly reflected $G$-BSDE with parameters $(\phi(X^{t,x}_T), \widetilde{f}^{t,x},\widetilde{g}^{t,x}, L^{t,x}, U^{t,x})$, where 
    \begin{align*}
    \widetilde{f}^{t,x}(s,y,z)=f(s,X^{t,x}_s,y,z)\mathbbm{1}_{[t,T]}(s) \quad \text{and} \quad  \widetilde{g}^{t,x}(s,y,z)=g(s,X^{t,x}_s,y,z)\mathbbm{1}_{[t,T]}(s).
    \end{align*}
    For each fixed $x\in\mathbb{R}$, suppose that $0\leq t_1\leq t_2\leq T$, by Proposition \ref{a prior estimate}, there exists a constant $C$ depending on $T,k,\kappa,G,x$, such that
	\begin{align*}
		&\hspace{-1cm}|u(t_1,x)-u(t_2,x)|^2=|Y^{t_1,x}_0-Y^{t_2,x}_0|^2\\
		\leq &C\bigg(\mathbb{E}\sup_{t\in[0,T]}|L^{t_1,x}_t-L^{t_2,x}_t|^2+\mathbb{E}\sup_{t\in[0,T]}|U^{t_1,x}_t-U^{t_2,x}_t|^2\bigg)^{\frac{1}{2}}+C\mathbb{E}|\phi(X_T^{t_1,x})-\phi(X_T^{t_2,x})|^2\\
		&+C\mathbb{E}\left[\int_0^T \Big|\widetilde{f}^{t_1,x}(s,X^{t_1,x}_s,Y^{t_2,x}_s,Z^{t_2,x}_s)-\widetilde{f}^{t_2,x}(s,X^{t_2,x}_s,Y^{t_2,x}_s,Z^{t_2,x}_s)\Big|^2ds\right]\\
        &+C\mathbb{E}\left[\int_0^T \Big|\widetilde{g}^{t_1,x}(s,X^{t_1,x}_s,Y^{t_2,x}_s,Z^{t_2,x}_s)-\widetilde{g}^{t_2,x}(s,X^{t_2,x}_s,Y^{t_2,x}_s,Z^{t_2,x}_s)\Big|^2ds\right].
	\end{align*}
	Note that
	\begin{align*}
		&\hspace{-0.5cm}\sup_{t\in[0,T]}|L^{t_1,x}_t-L^{t_2,x}_t|\\
		\leq&|h(t_1,x)-h(t_2,x)|+\sup_{t\in[t_1,t_2]}|h(t,X_t^{t_1,x})-h(t_2,x)|+\sup_{t\in[t_2,T]}|h(t,X^{t_1,x}_t)-h(t,X^{t_2,x}_t)|\\
		\leq &2\sup_{t\in[t_1,t_2]}|h(t,x)-h(t_2,x)|+\sup_{t\in[t_1,t_2]}|h(t,X^{t_1,x}_t)-h(t,x)|+\sup_{t\in[t_2,T]}\kappa|X^{t_1,x}_t-X^{t_2,x}_t|\\
		\leq & 2\sup_{t\in[t_1,t_2]}|h(t,x)-h(t_2,x)|+\sup_{t\in[t_1,t_2]}\kappa|X^{t_1,x}_t-x|+\sup_{t\in[t_2,T]}\kappa|X^{t_2,X^{t_1,x}_{t_2}}_t-X^{t_2,x}_t|.
	\end{align*}
	Letting $\delta=t_2-t_1$, by Proposition \ref{the1.17}, we have
	\begin{displaymath}
		\lim_{\delta\rightarrow 0}\mathbb{E}\sup_{t\in[0,T]}|L^{t_1,x}_t-L^{t_2,x}_t|^2=0.
	\end{displaymath}
	A similar analysis yields that
	\begin{align*}
		\lim_{\delta\rightarrow 0}\mathbb{E}\sup_{t\in[0,T]}|U^{t_1,x}_t-U^{t_2,x}_t|^2=0\quad\text{and}\quad
		\lim_{\delta\rightarrow 0}\mathbb{E}|\phi(X^{t_1,x}_T)-\phi(X^{t_2,x}_T)|^2=0.
	\end{align*}
	By simple calculation, we obtain that
	\begin{align*}
		&\hspace{-0.5cm}\int_0^T \Big|\widetilde{f}^{t_1,x}(s,X^{t_1,x}_s,Y^{t_2,x}_s,Z^{t_2,x}_s)-\widetilde{f}^{t_2,x}(s,X^{t_2,x}_s,Y^{t_2,x}_s,Z^{t_2,x}_s)\Big|^2ds\\
		\leq &C \int_{t_1}^{t_2}\Big(|f(s,0,0,0)|^2+|X_s^{t_1,x}|^{2k+2}+|X_s^{t_1,x}|^2+|Y_s^{t_2,x}|^2+|Z_s^{t_2,x}|^2\Big)ds\\
		&+C\int_{t_2}^T \Big(1+|X_s^{t_1,x}|^k+|X_s^{t_2,x}|^k\Big)^2|X_s^{t_1,x}-X^{t_2,x}_s|^2ds.
	\end{align*}
	Noting that the process $\{\int_0^t |Z_s^{t_2,x}|^2 ds\}_{t\in[0,T]}\in S_G^1(0,T)$, applying Propositions \ref{uniformcontinuous} and \ref{the1.17}, we have
	\begin{displaymath}
		\lim_{\delta\rightarrow 0}\mathbb{E}\left[\int_0^T \Big|\widetilde{f}^{t_1,x}(s,X^{t_1,x}_s,Y^{t_2,x}_s,Z^{t_2,x}_s)-\widetilde{f}^{t_2,x}(s,X^{t_2,x}_s,Y^{t_2,x}_s,Z^{t_2,x}_s)\Big|^2ds\right]=0.
	\end{displaymath}
    Similarly, we have
    \begin{displaymath}
		\lim_{\delta\rightarrow 0}\mathbb{E}\left[\int_0^T \Big|\widetilde{g}^{t_1,x}(s,X^{t_1,x}_s,Y^{t_2,x}_s,Z^{t_2,x}_s)-\widetilde{g}^{t_2,x}(s,X^{t_2,x}_s,Y^{t_2,x}_s,Z^{t_2,x}_s)\Big|^2ds\right]=0.
	\end{displaymath}
	All the above analyses imply that $u$ is continuous in $t$. The proof is complete.
\end{proof}

Our main result in this section is that the function $u$ defined in \eqref{eq1.8} is the solution to the following fully nonlinear obstacle problem:
\begin{equation}\label{eq1.7}
	\begin{cases}
		\max\Big(u-h',\,\min\big(-\partial_t u(t,x)-F(D_x^2 u,D_x u,u,x,t),u-h\big)\Big)=0,
		&\\
		u(T,x)=\phi(x),  &
	\end{cases}
\end{equation}
where
\begin{align*}
	F(D_x^2 u,D_x u,u,x,t)=&G(H(D_x^2 u,D_x u,u,x,t))+ b(t,x)D_x u+f(t,x,u,\sigma(t,x)D_x u),\\
	H(D_x^2 u,D_x u,u,x,t)=&\sigma^2(t,x) D_x^2u+2l(t,x)D_xu+2g(t,x,u,\sigma(t,x)D_x u).
\end{align*}
Note that by Lemma \ref{l1} and Lemma \ref{l2}, one only obtains that $u$ is a continuous function but it may be not differentiable. The notion of viscosity solutions was introduced by P.-L. Lions and M. Crandall independently in the 1980s, primarily to address issues arising in nonlinear PDEs, such as the lack of classical solutions due to singularities or non-smoothness. Viscosity solutions have since become a fundamental tool in the analysis of various types of nonlinear PDEs, including Hamilton-Jacobi equations, obstacle problems, and certain types of evolution equations. We begin by providing the definition of a viscosity solution to \eqref{eq1.7}, which relies on the concepts of sub-jets and super-jets. Further elaboration can be found in the paper \cite{crandall1992user}. 

\begin{definition}
	Let $u\in C((0,T)\times\mathbb{R})$ and $(t,x)\in(0,T)\times \mathbb{R}$. We denote $\mathcal{P}^{2,+} u(t,x)$ as the ``parabolic superjet" of $u$ at $(t,x)$, which comprises triples $(p,q,X)\in\mathbb{R}^3$ such that
	\begin{align*}
		u(s,y)\leq u(t,x)+p(s-t)+ q(y-x)+\frac{1}{2} X(y-x)^2+o(|s-t|+|y-x|^2).
	\end{align*}
	Similarly, we define $\mathcal{P}^{2,-} u(t,x)$  as the ``parabolic subjet" of $u$ at $(t,x)$ by $$\mathcal{P}^{2,-} u(t,x):=-\mathcal{P}^{2,+}(- u)(t,x).$$
\end{definition}


\begin{definition}
	Let $u$ be a continuous function defined on $[0,T]\times \mathbb{R}$.
	\begin{enumerate}[label=(\roman*)]
	\item  It is called a viscosity subsolution  of \eqref{eq1.7} if $u(T,x)\leq \phi(x)$ for $x\in\mathbb{R}$, and at any point $(t,x)\in(0,T)\times\mathbb{R}$, for any $(p,q,X)\in\mathcal{P}^{2,+}u(t,x)$,
	\begin{displaymath}
		\max\Big(u(t,x)-h'(t,x), \,\min\big(u(t,x)-h(t,x), -p-F(X,q,u(t,x),x,t)\big)\Big)\leq 0;
	\end{displaymath}
	\item  It is called a viscosity supersolution of \eqref{eq1.7} if $u(T,x)\geq\phi(x)$ for $x\in\mathbb{R}$, and at any point $(t,x)\in(0,T)\times\mathbb{R}$, for any $(p,q,X)\in\mathcal{P}^{2,-}u(t,x)$,
	\begin{displaymath}
		\max\Big(u(t,x)-h'(t,x),\,\min\big(u(t,x)-h(t,x), -p-F(X,q,u(t,x),x,t)\big)\Big)\geq 0;
	\end{displaymath}
	\item  It is called a viscosity solution of \eqref{eq1.7} if it is both a viscosity subsolution and supersolution.
	\end{enumerate}
\end{definition}

Denote by $\{(Y_s^{n,t,x},Z_s^{n,t,x},A_s^{n,t,x})\}_{s\in[t,T]}$ the solution of the following penalized reflected $G$-BSDEs:
\begin{align*}
	\begin{cases}
		Y_s^{n,t,x}= \phi(X_T^{t,x})+\int_s^T f(r,X_r^{t,x},Y_r^{n,t,x},Z_r^{n,t,x})dr+\int_s^T g(r,X_r^{t,x},Y_r^{n,t,x},Z_r^{n,t,x})d\langle B\rangle_r\\
		\ \ \ \ \ \ \ \ \ \ \ -n\int_s^T (Y_r^{n,t,x}-h'(r,X_r^{t,x}))^+dr-\int_s^T Z_r^{n,t,x}dB_r+(A_T^{n,t,x}-A_s^{n,t,x}),\hspace{0.5cm} t\leq s\leq T,\vspace{0.2cm}\\
		Y_s^{n,t,x,}\geq h(s,X^{t,x}_s), \hspace{0.5cm} t\leq s\leq T,\vspace{0.2cm}\\
		\{\int_t^s (h(r,X^{t,x}_r)-Y^{n,t,x}_r)dA^{n,t,x}_r\}_{s\in[t,T]} \textrm{ is a non-increasing $G$-martingale}.
\end{cases}
\end{align*}
By Theorem \ref{main}, $Y^{t,x}$ is the limit of $Y^{n,t,x}$ as $n$ goes to infinity. We define
\begin{displaymath}
	u_n(t,x):=Y_t^{n,t,x}, \quad (t,x)\in[0,T]\times \mathbb{R}.
\end{displaymath}
By Theorem 6.7 in \cite{li2018reflected}, $u_n$ is the viscosity solution of the following parabolic PDE:
\begin{equation}\label{eq1.9}
	\begin{cases}
		\min\Big(u_n(t,x)-h(t,x),\,-\partial_t u_n-F_n(D_x^2 u_n,D_x u_n,u_n,x,t)\Big)=0, & (t,x)\in[0,T]\times\mathbb{R}\vspace{0.2cm}\\
		u_n(T,x)=\phi(x), &  x\in\mathbb{R},
\end{cases}\end{equation}
where
\begin{displaymath}
	F_n(D_x^2 u,D_x u,u,x,t)=F(D_x^2 u,D_x u,u,x,t)-n(u(t,x)-h'(t,x))^+.
\end{displaymath}

\begin{theorem}\label{the1.21}
	Under Conditions (Ai)-(Aiii), the function $u$ defined in \eqref{eq1.8} is the unique viscosity solution of the double obstacle problem \eqref{eq1.7}.
\end{theorem}

\begin{proof}
	By Theorem \ref{main}, for each $(t,x)\in[0,T]\times\mathbb{R}$, we have
	\begin{displaymath}
		u_n(t,x)\downarrow u(t,x),\quad \textrm{ as }n\uparrow \infty.
	\end{displaymath}
	Note that $u_n$ is continuous by Lemmas 6.4-6.6 in \cite{li2018reflected}. Since $u$ is also continuous, applying Dini's theorem yields that the sequence $u^n$ uniformly converges to $u$ on compact sets. The proof will proceed in the following two steps.\\
	
\noindent	\textbf{Step 1: Viscosity subsolution}. For each fixed $(t,x)\in(0,T)\times\mathbb{R}$, let $(p,q,X)\in\mathcal{P}^{2,+} u(t,x)$. Suppose that $u(t,x)=h(t,x)$. Noting that $u(t,x)\leq h'(t,x)$, it is easy to check that
	\begin{displaymath}
		\max\Big(u(t,x)-h'(t,x),\, \min\big(u(t,x)-h(t,x), -p-F(X,q,u(t,x),x,t)\big)\Big)\leq 0.
	\end{displaymath}
	Assume that $u(t,x)>h(t,x)$. It remains to prove that
	\begin{displaymath}
		-p-F(X,q,u(t,x),x,t)\leq 0.
	\end{displaymath}
	By Lemma 6.1 in \cite{crandall1992user}, there exist sequences
	\begin{displaymath}
		n_j\rightarrow \infty,\qquad (t_j,x_j)\rightarrow (t,x),\qquad (p_j,q_j,X_j)\in \mathcal{P}^{2,+} u_{n_j}(t_j,x_j),
	\end{displaymath}
	such that $(p_j,q_j,X_j)\rightarrow(p,q,X)$. Recalling that $u_n$ is the viscosity solution to Equation \eqref{eq1.9}, hence a subsolution, we have, for any $j$,
    \begin{displaymath}
		-p_j-F(X_j,q_j,u_{n_j}(t_j,x_j),x_j,t_j)+n_{j}(u_{n_j}(t_j,x_j)-h'(t_j,x_j))^+\leq 0.
	\end{displaymath}
    Thus,
	\begin{displaymath}
		-p_j-F(X_j,q_j,u_{n_j}(t_j,x_j),x_j,t_j)\leq 0.
	\end{displaymath}
	Letting $j$ go to infinity in the above inequality yields the desired result.
	Therefore,  $u$ is a subsolution of the fully nonlinear obstacle problem \eqref{eq1.7}.\\
	
	\noindent	\textbf{Step 2: Viscosity supersolution.} For each fixed $(t,x)\in(0,T)\times\mathbb{R}$, and $(p,q,X)\in\mathcal{P}^{2,-} u(t,x)$. It is sufficient to show that when $u(t,x)<h'(t,x)$,
	\begin{displaymath}
		-p-F(X,q,u(t,x),x,t)\geq 0.
	\end{displaymath}
	Applying Lemma 6.1 in \cite{crandall1992user} again, there exist sequences
	\begin{displaymath}
		n_j\rightarrow \infty,\quad (t_j,x_j)\rightarrow (t,x),\quad (p_j,q_j,X_j)\in \mathcal{P}^{2,-} u_{n_j}(t_j,x_j),
	\end{displaymath}
	such that $(p_j,q_j,X_j)\rightarrow(p,q,X)$. Since $u_n$ is the viscosity solution to Equation \eqref{eq1.9}, hence a supersolution, we derive that for any $j$,
	\begin{displaymath}
		-p_j-F(X_j,q_j,u_{n_j}(t_j,x_j),x_j,t_j)+n_{j}(u_{n_j}(t_j,x_j)-h'(t_j,x_j))^+\geq 0.
	\end{displaymath}
	Given that 	$u_n$ converges uniformly on compact sets, for $j$ sufficiently large,  $u_{n_j}(t_j,x_j)<h'(t_j,x_j)$ under the assumption that $u(t,x)<h'(t,x)$.  Therefore, as $j$ tends to infinity, the above inequality implies that
	\begin{displaymath}
		-p-F(X,q,u(t,x),x,t)\geq 0,
	\end{displaymath}
	which is the desired result. Consequently, 
	$u$ is a viscosity solution of \eqref{eq1.7}.
	
Following a similar analysis as the proof of Theorem 6.3 in \cite{hamadene2005bsdes}, we can establish the uniqueness of the viscosity solution to the fully nonlinear obstacle problem \eqref{eq1.7}. This concludes the proof.
\end{proof}

\section*{Acknowledgments}
	The research of Li was supported  by the National Natural Science Foundation of China (No. 12301178), the Natural Science Foundation of Shandong Province for Excellent Young Scientists Fund Program (Overseas) (No. 2023HWYQ-049) and  the Qilu Young Scholars Program of Shandong University.

\section*{Conflict of Interest}
The authors declared that there is no conflict of interest.

\bibliography{bib-ms}

\end{document}